\documentclass{article}

\usepackage{indentfirst}
\usepackage{PRIMEarxiv}
\usepackage{amsmath,amssymb,amsthm}
\usepackage[ruled,vlined]{algorithm2e}
\usepackage[utf8]{inputenc} 
\usepackage[T1]{fontenc}    
\usepackage{hyperref}       
\usepackage{url}            
\usepackage{booktabs}       
\usepackage{amsfonts}       
\usepackage{nicefrac}       
\usepackage{microtype}      
\usepackage{lipsum}
\usepackage{graphicx}
\usepackage{color}
\usepackage{subfigure}
\usepackage{authblk}
\graphicspath{{media/}}     

\title{Symplectic Extra-gradient Type Method for Solving General Non-monotone Inclusion Problem
\thanks{This work was partially supported by grant 12288201 from NSF of China.
}}

\author[1]{Ya-xiang Yuan}
\author[ 1,2]{Yi Zhang\thanks{Corresponding author: zhangyi2020@lsec.cc.ac.cn}  }
\affil[1]{Institute of Computational Mathematics and Scientific/Engineering Computing, Academy of Mathematics and Systems
	Science, Chinese Academy of Sciences, Beijing 100190, China}
\affil[2]{University of Chinese Academy of Sciences, Beijing 100049, China}

\newcommand{\inner}[1]{\left\langle#1\right\rangle}
\newcommand{\norm}[1]{\left\|#1\right\|}
\newcommand{\E}{\mathcal{E}}

\newcommand{\tr}{^\mathrm{T}}

\newcommand{\G}{\mathcal{G}}
\newcommand{\expect}[1]{\mathbb{E}\left[#1\right]}
\newcommand{\hh}{\mathcal{H}}
\newtheorem{theorem}{Theorem}
\newtheorem{lemma}[theorem]{Lemma}
\newtheorem{corollary}[theorem]{Corollary}
\newtheorem{proposition}{Proposition}
\newtheorem{example}{Example}
\newcommand{\dist}{\mathrm{dist}}
\newcommand{\zero}{\mathrm{zero}}

\begin{document}
	\maketitle

\begin{abstract}
	In recent years, accelerated extra-gradient methods have attracted much attention by researchers, for solving monotone inclusion problems. A limitation of most current accelerated extra-gradient methods lies in their direct utilization of the initial point, which can potentially decelerate numerical convergence rate. In this work, we present a new accelerated extra-gradient method, by utilizing the symplectic acceleration technique. We establish the inverse of quadratic convergence rate by employing the Lyapunov function technique. Also, we demonstrate a faster inverse of quadratic convergence rate alongside its weak convergence property under stronger assumptions. To improve practical efficiency, we introduce a line search technique for our symplectic extra-gradient method. Theoretically, we prove the convergence of the symplectic extra-gradient method with line search. Numerical tests show that this adaptation exhibits faster convergence rates in practice compared to several existing extra-gradient type methods.
\end{abstract}


\section{Introduction}
Min-max problem and min-max duality theory lie at the foundations of game theory, designing algorithms \cite{valkonen14, zhu08} and duality theory of mathematical programming \cite{bonnas00}, and have found far-reaching applications across a range of disciplines, including decision theory \cite{myerson91}, economics \cite{neumann04}, structural design \cite{taylor84}, control theory \cite{vasilyev10} and robust optimization \cite{bental09}. Recently, as burgeoning of Generative Adversarial Networks \cite{gidel18,goodfellow14} and adversarial attacks \cite{madry18}, solving min-max problem under nonconvex-nonconcave assumption has gained researchers' attention. However, due to the nonconvex-nonconcave assumption,  solving min-max problem exactly is nearly impossible. Similar to nonlinear programming, solving first-order stationary point of nonconvex-nonconcave min-max problem is relatively easier. The first-order stationary condition of most min-max problems can be described by the following inclusion problem
\begin{equation}
	0\in T(z):=F(z)+G(z),\quad z\in\hh,
	\label{eq:main}
\end{equation}
where $\hh$ is a real Hilbert space equipped with inner product $\inner{\cdot, \cdot}$ and norm $\norm{\cdot}=\sqrt{\inner{\cdot, \cdot}}$. Let $\zero(T)$ be the set of all solutions of \eqref{eq:main}. Throughout this paper, we assume $F: \hh\to\hh$ is a single-value $L-$Lipschitz continuous operator, $G: \hh\to 2^{\hh}$ is a set-valued maximally monotone operator and $\zero(T)$ is non-empty.

When $G=0$, $F$ is monotone and $L-$Lipschitz continuous, the extra-gradient (EG) method proposed in \cite{korpelevich76} and \cite{popov80} is an efficient method for solving \eqref{eq:main}. Theoretically, the EG method has $O(1/k)$ ergodic convergence rate without additional requirement on $F$. Practically, the EG method can be described as applying gradient descent step twice per iteration. Owning to its fascinating theoretical properties and simple recursive rule, extensions of the EG method have been studied, such as the mirror-prox method in \cite{juditsky11} and \cite{nemirovski04} when $G$ is the normal cone of convex set $C$, the Tseng's splitting method or forward-backward splitting method in \cite{raguet13} and \cite{tseng00} when $G\ne 0$ and the EG+ method \cite{diakonikolas21} to solve \eqref{eq:main} under non-monotone assumption. The adaptive EG+ method proposed in \cite{fan23}  and \cite{pethick22} further generalizes the EG+ method by considering an adaptive step-size. Here, we refer to the survey \cite{tran23} of the EG method for more existing theory about the EG method. As stochastic optimization continues to gain significant attention, research on stochastic extra-gradient method has seen a rise, such as \cite{gorbunov22}, \cite{iusem17}, \cite{kannan19} and \cite{mishchenko20}.

Nesterov's accelerated gradient method \cite{nesterov83} has stimulated research on acceleration technique. The first acceleration form of EG method for solving monotone inclusion problem, called extra anchored gradient (EAG) algorithm, was presented by \cite{yoon21}. From the observation that the EG method is closed to proximal point algorithm \cite{mokhtari20}, \cite{yoon21} uses the Halpern's iteration \cite{diakonikolas21,halpern67,qi21}, which is a widely studied acceleration technique for the proximal point algorithm, to derive the EAG algorithm. Also, the complexity lower bound of EG type method was proven to be $O(1/k^2)$ in \cite{yoon21}. Nowadays, studies of acceleration form of the EG method have increased. \cite{lee21} combined the EAG algorithm and the EG+ method and obtained the fast EG method, which can solve \eqref{eq:main} under non-monotone assumption. \cite{cai22} obtained an accelerated forward-backward splitting method by further generalizing the fast EG method.

However, the EAG type method is prone to oscillatory phenomena, which eventually slows down numerical performance. That is because the Halpern's iteration can be characterized as calculating convex combination of the initial starting point and the current feasible point. The use of the initial starting point may cause the current feasible point far away from the solution set, leading to oscillation phenomena. 
\subsection{Our Contributions}
In Section \ref{sec:3}, we exploit a recently proposed acceleration method, called \emph{symplectic acceleration} \cite{yi23}, to devise a novel accelerated variant of the extra-gradient method, named \emph{Symplectic Extra-Gradient (SEG) Method}. Capitalizing on the Lyapunov function framework proposed in \cite{yi23}, we establish that both the SEG and its extended versions exhibit a convergence rate of $O(1/k^2)$. Moreover, in Section \ref{sec:4}, we demonstrate that by imposing stricter conditions, we can prove the SEG type methods admits a faster $o(1/k^2)$ convergence rate and weak convergence property. To the best of our knowledge, the SEG type method is the first EG type method with proved  $o(1/k^2)$ convergence rate. A concise summary of these theoretical results is presented in Table \ref{tab:result}.

\begin{table}[ht]
	\centering
	\begin{tabular}{c|c|c}
		\toprule
		Methods&Convergence Rate& Weak Convergence Property\\
		\midrule
		EG+ \cite{diakonikolas21}&$O(1/k)$&-\\
		FEG \cite{lee21}&$O(1/k^2)$&-\\
		\textcolor{red}{SEG+ (This paper)}&\textcolor{red}{$o(1/k^2)$}&\textcolor{red}{\checkmark}\\
		\bottomrule
	\end{tabular}
	\caption{Illustration of theoretical results in this paper.}
	\label{tab:result}
\end{table}

In Section \ref{sec:5}, we address the challenge posed by the difficulty in precisely estimating the Lipschitz constant and the comonotone index, as well as handling situations involving rapid variations in both the local Lipschitz constant and the local comonotone index. To this end, we integrate the line search framework into the SEG framework. Under specified assumptions,  we prove  the convergence of the SEG method equipped with line search. Our numerical experiments demonstrate that our method performs better than several existing EG type methods.
\section{Preliminaries}
\subsection{Basic Concepts}
Let $\hh$ be a real Hilbert space, and let $T: \hh\to 2^{\hh}$ is a set-valued operator. The operator $T$ is said to be \emph{$\rho-$comonotone} if \[
\inner{u-v,x-y}\geqslant\rho\norm{u-v}^2,\quad \forall x,\  y\in\hh,\  u\in T(x),\  v\in T(y).
\]
When $\rho=0$, the concept of $0-$comonotone coincide with the concept of monotone. If $\rho>0$,  $\rho-$comonotone is the same as $\rho-$cocoercive. A monotone operator $G$ is \emph{maximally} if the graph of $G$ is not a proper subset of the graph of another monotone operator. The Minty surjectivity theorem \cite{minty62} shows that a monotone operator $G$ is maximal if and only if the domain of resolvent $J_G=(I+G)^{-1}$ of $G$ is $\hh$. Also, a continuous and monotone single-value operator is a maximally monotone operator. For further theories on comonotone operator, we refer to \cite{bauschke21}.
\subsection{Extra-gradient Method and Extra-gradient+ Method}
The extra-gradient method for solving zero-point of $L-$Lipschitz continuous and monotone operator $F$, i. e. solving\[
0=F(z),
\] 
is given as follows:\begin{equation}
	\begin{split}
		z_{k+\frac{1}{2}}&=z_k-sF(z_k),\\
		z_{k+1}&=z_k-sF(z_{k+\frac{1}{2}}).
	\end{split}
	\label{eq:eg}
\end{equation}
In \cite{korpelevich76} and \cite{tseng95}, the proof of convergence results of EG method is based on the following inequality:\[
\norm{z_{k+1}-z^*}^2\leqslant\norm{z_k-z^*}^2-(1-s^2L^2)\norm{F(z_k)}^2,\quad\forall 0=F(z^*).
\]
By summing the above inequality respect to $k$, one can easily show that the ergodic convergence rate of the EG method on $\norm{F(z_k)}^2$ is $O(1/k)$ when $0<s<\dfrac{1}{L}$. The result that the sequence $\{z_k\}$ converges weakly to a zero-point of $F$ relies on the following propositions.
\begin{proposition}[Lemma 2.47 in \cite{bauschke17}]
	\label{prop:discreteopial}
	Let $\{z_k\}$ be a sequence in $\hh$ and let $C$ be a nonempty subset of $\hh$. Suppose that, for every $z\in C$, $\norm{z_k-z}$ converges and that
	every weak sequential cluster point of $\{z_k\}$ belongs to $C$. Then $\{z_k\}$		converges weakly to a point in $C$.
\end{proposition}
\begin{proposition}[Proposition 20.38 (ii) in \cite{bauschke17}]
	\label{prop:weakclosed}
	Let $T: \hh\to 2^{\hh}$ be a maximally monotone operator, $\{z_k\}$ and $\{v_k\}$ be the sequences in $\hh$ such that $v_k\in T(z_k)$. If $z_k\rightharpoonup z$ and $v_k\to v$. Then $v\in T(z)$.
\end{proposition}
The details of the proof of weak convergence property of the EG method can be found in \cite{nadezhkina06}. The EG+ method proposed in \cite{diakonikolas21} for solving zero-point of $L-$Lipschitz continuous and $\rho-$comonotone operator $F$ is given as follows:\begin{equation}
	\begin{split}
		z_{k+\frac{1}{2}}&=z_k-\frac{s}{\beta}F(z_k),\\
		z_{k+1}&=z_k-sF(z_{k+\frac{1}{2}}).
	\end{split}
	\label{eq:eg+}
\end{equation}
Similar to the EG method, \cite{diakonikolas21} shows that if $\beta=\dfrac{1}{2}, s=\dfrac{1}{2L}$, $\rho>-\dfrac{1}{8L}$, then we have\[
\norm{z_{k+1}-z^*}^2\leqslant\norm{z_k-z^*}-\frac{1}{4L}\left(\frac{1}{4L}+2\rho\right)\norm{F(z_{k+\frac{1}{2}})}^2,\quad\forall 0=F(z^*).
\]
By using the above inequality respect to $k$, one can prove the ergodic convergence rate of the EG+ method.
\subsection{Anchor Acceleration}
Since \cite{mokhtari20} suggested that the sequence generated by the EG method is closed to the sequence generated by the proximal point algorithm, recent study of acceleration form of the EG method is based on existing acceleration technique for proximal point algorithm. Among most acceleration form of proximal point algorithm, the Halpern's iteration \cite{halpern67} is one of the most wildly studied acceleration technique. The Halpern's iteration for solving zero-point of maximally monotone operator $F$ is
\begin{equation}
	z_{k+1}=\alpha_kz_0+(1-\alpha_k)J_F(z_k).
	\label{eq:halpern}
\end{equation}
Replacing the proximal point term by the EG step, we obtain the extra anchored gradient algorithm \cite{yoon21}:\begin{equation}
	\begin{split}
		z_{k+\frac{1}{2}}&= \alpha_kz_0+(1-\alpha_k)z_k-\beta_kF(z_k),\\
		z_{k+1}&= \alpha_kz_0+(1-\alpha_k)z_k-\beta_kF(z_{k+\frac{1}{2}}).
	\end{split}
	\label{eq:anchoracceleration}
\end{equation}
In \cite{yoon21}, \eqref{eq:anchoracceleration} was proven to be $O(1/k^2)$ convergence rate if $F$ is an $L-$Lipschitz continuous and monotone operator and the parameters in \eqref{eq:anchoracceleration} satisfy several assumptions. Combining the EG+ method and \eqref{eq:anchoracceleration}, \cite{lee21} gave the fast EG (FEG) method as follows:\begin{equation}
	\begin{split}
		z_{k+\frac{1}{2}}&= \alpha_kz_0+(1-\alpha_k)z_k-(1-\alpha_k)\left(\frac{1}{L}+2\rho\right)F(z_k),\\
		z_{k+1}&= \alpha_kz_0+(1-\alpha_k)z_k-\frac{1}{L}F(z_{k+\frac{1}{2}})-(1-\alpha_k)2\rho F(z_k).
	\end{split}
	\label{eq:feg}
\end{equation}
If $\alpha_k=\dfrac{1}{k+1}$, \eqref{eq:feg} admits $O(1/k^2)$ convergence rate. The convergence results of accelerated forward-backward splitting method in \cite{cai22} is similar to the convergence results \eqref{eq:feg}.
\subsection{Symplectic Acceleration}
\label{sec:2.4}
The symplectic acceleration technique proposed in \cite{yi23} is a new way to accelerate proximal point algorithm. The symplectic proximal point algorithm for solving zero-point of maximally monotone operator $F$ is given as follows:\begin{equation}
	\begin{split}
		\tilde{z}_{k+1}&=\frac{k}{k+r}z_k+\frac{r}{k+r}u_k,\\
		z_{k+1}&=J_F(z_k),\\
		u_{k+1}&=u_k+\frac{D}{r}(z_{k+1}-\tilde{z}_{k+1}).
	\end{split}
	\label{eq:sppa}
\end{equation} 
The reason why \eqref{eq:sppa} is called symplectic proximal point algorithm is that \eqref{eq:sppa} is obtained by applying  first-order implicit symplectic method to a first-order ODEs system. If $r>1$, $0<D< r-1$, then \eqref{eq:sppa} admits $O(1/k^2)$ convergence rate. Moreover, $o(1/k^2)$ convergence rate and weak convergence property of \eqref{eq:sppa} can be established by considering two auxiliary functions.

\section{Symplectic Extra-gradient Type Method}
\label{sec:3}

In this section, we introduce the symplectic extra-gradient+ framework for solving general inclusion problem\[
0\in T(z):=F(z)+G(z).
\]
At first, we explain how to propose symplectic extra-gradient framework by considering the zero-point problem $0=F(z)$. In Section \ref{sec:2.4}, we introduce the symplectic proximal point algorithm \eqref{eq:sppa}. \eqref{eq:sppa} can be generalized as follows:\begin{equation}
	\label{eq:sppageneral}
	\begin{split}
		\tilde{z}_{k+1}&=(1-\alpha_k) z_k+\alpha_k u_k,\\
		z_{k+1}&=J_F(\tilde{z}_{k+1}),\\
		u_{k+1}&=u_k+C_k(z_{k+1}-\tilde{z}_{k+1}).
	\end{split}
\end{equation}
To formulate the desired algorithm, we need a two-step modification to equation \eqref{eq:sppageneral}. Due to the fact that $\tilde{z}_{k+1}-z_{k+1}= F(z_{k+1})$, the update rule for $u_{k+1}$ is transformed to\[
u_{k+1}=u_k-C_kF(z_{k+1}).
\] Moreover, similar to \eqref{eq:anchoracceleration}, we replace the proximal point step by the EG step and obtain the \emph{symplectic extra-gradient (SEG) method} described as follows:
\begin{subequations}
	\begin{align}
		\tilde{z}_{k+1}&=(1-\alpha_k) z_k+\alpha_k u_k,\\
		z_{k+\frac{1}{2}}&=\tilde{z}_{k+1}-\beta_kF(z_k),\label{eq:generalsegb}\\
		z_{k+1}&=\tilde{z}_{k+1}-\beta_kF(z_{k+\frac{1}{2}}),\\
		u_{k+1}&=u_k-C_kF(z_{k+1}).
	\end{align}
	\label{eq:seg}
\end{subequations}
The coefficients $\{\alpha_k\}, \{\beta_k\}$ in \eqref{eq:seg} are positive for all $k\geqslant 1$. First, we notice that if we does not impose any constrains on $\{\alpha_k\}$ and $\{C_k\}$, the EG method and EAG algorithm can be included in \eqref{eq:seg}. Specifically, let $\alpha_k\equiv 1, \beta_k\equiv\dfrac{1}{L}$, \eqref{eq:seg} can be simplified as \eqref{eq:eg}, and let $C_k\equiv 0$, \eqref{eq:seg} is the same as \eqref{eq:anchoracceleration}.  However, it is noteworthy that despite the apparent requirement of three gradient evaluations per iteration in \eqref{eq:seg}, a strategic reordering of the iterative process can economize computations to merely two gradient evaluations per loop. This is achieved through the adoption of the following alternative iteration rule:
\begin{align*}
	u_{k+1}&=u_k-C_kF(z_k),\\
	\tilde{z}_{k+1}&=(1-\alpha_k) z_k+\alpha_k u_{k+1},\\
	z_{k+\frac{1}{2}}&=\tilde{z}_{k+1}-\beta_kF(z_k),\\
	z_{k+1}&=\tilde{z}_{k+1}-\beta_kF(z_{k+\frac{1}{2}}).
\end{align*}
The employment of such a reordering methodology can decrease the quantity of necessary gradient evaluations within subsequent algorithms. Consequently, further expounding on this technique will be omitted in the remaining text.

It should be noted that the SEG method is only suitable for solving zero-point of $F$ under monotone assumption, coupled with complicate convergence results discussed in Section \ref{app:segrate}. As such, we focus on the symplectic extra-gradient+ framework in the remaining text. Inspired by the FEG method \eqref{eq:feg} proposed in \cite{lee21}, the corresponding \emph{symplectic extra-gradient+ (SEG+) method}  for solving zero-point of $L-$Lipschitz continuous and $\rho-$comonotone operator $F$ is given as follows:\begin{subequations}
	\begin{align}
		\tilde{z}_{k+1}&=(1-\alpha_k) z_k+\alpha_k u_k,\\
		z_{k+\frac{1}{2}}&=\tilde{z}_{k+1}-(1-\alpha_k)\left(\frac{1}{L}+2\rho\right)F(z_k),\\
		z_{k+1}&=\tilde{z}_{k+1}-\dfrac{1}{L}F(z_{k+\frac{1}{2}})-(1-\alpha_k)2\rho F(z_k),\\
		u_{k+1}&=u_k-C_kF(z_{k+1}).
	\end{align}
	\label{eq:seg+}
\end{subequations}

Motivated by Tseng's splitting algorithm \cite{tseng00} and the accelerated forward-backward splitting method \cite{cai22}, we can further generalize the SEG+ method and propose the \emph{symplectic forward-backward splitting (SFBS) method} with $\alpha_k=\dfrac{r}{k+r}$ and $C_k=\dfrac{D}{r}$, which is described in Algorithm \ref{al:sfbsm}.
\begin{algorithm}[ht]
	\label{al:sfbsm}
	\caption{Symplectic Forward-Backward Splitting Method, SFBS}
	\textbf{Input: } Operators $F$ and $G$, $L\in(0,+\infty), \rho\in\left(-\dfrac{1}{2L},+\infty\right)$\;
	\textbf{Initialization: }$z_0$, $u_0=z_0$, $\tilde{G}_0=0$\;
	\For{$k=0, 1, \cdots$}{
		$\tilde{z}_{k+1}=\dfrac{k}{k+r}z_k+\dfrac{r}{k+r}u_k$\;
		$z_{k+\frac{1}{2}}=\tilde{z}_{k+1}-\dfrac{k}{k+r}\left(\dfrac{1}{L}+2\rho\right)(F(z_k)+\tilde{G}_k)$\;
		$z_{k+1}=J_{L^{-1}G}\left(\tilde{z}_{k+1}-\dfrac{1}{L}F(z_{k+\frac{1}{2}})-\dfrac{2\rho k}{k+r}(F(z_k)+\tilde{G}_k)\right)$\;
		$\tilde{G}_{k+1}=L\left[\tilde{z}_{k+1}-z_{k+1}-\dfrac{1}{L}F(z_{k+\frac{1}{2}})-\dfrac{2\rho k}{k+r}(F(z_k)+\tilde{G}_k)\right]$\;
		$u_{k+1}=u_k-\dfrac{D}{r}(F(z_{k+1})+\tilde{G}_{k+1})$.
	}
\end{algorithm}

Next, we apply the Lyapunov function technique to demonstrate the convergence rates of \eqref{eq:seg+}. Here, inspired by the Lyapunov function in \cite{lee21} for proving the convergence rates of the FEG method and the Lyapunov function in \cite{yi23} for proving the convergence rates of the symplectic proximal point algorithm, the Lyapunov function for Algorithm \ref{al:sfbsm} is \begin{equation}
	\begin{split}
		\E(k):=&\ \left[\frac{Dk^2}{2L}+\rho Dk(k-r)\right]\norm{F(z_k)+\tilde{G}_k}^2\\
		&\ +Drk\inner{F(z_k)+\tilde{G}_k,z_k-u_k}+\frac{r^3-r^2}{2}\norm{u_k-z^*}^2,
	\end{split}
	\label{lya:seg+}
\end{equation}
where $z^*\in\zero(T)$. Given the fundamental role that the difference of the Lyapunov function, i. e. $\E(k+1)-\E(k)$, plays in establishing convergence rates for \eqref{eq:seg+}, it becomes imperative to estimate the upper bound of the difference. By estimating the upper bound of $\E(k+1)-\E(k)$, we can give the convergence results of Algorithm \ref{al:sfbsm}  in Theorem \ref{thm:sfbsmrate}.
\begin{theorem}
	\label{thm:sfbsmrate}
	Let $\{z_k\}$ and $\{u_k\}$ be the sequences generated by Algorithm \ref{al:sfbsm}. If $r>1$ and $0<D\leqslant(r-1)\left(\dfrac{1}{L}+2\rho\right)$, then we have
	\[
	\left(\frac{1}{2L}+\rho-\frac{D}{2(r-1)}\right)k^2\dist(0, T(z_k))^2\leqslant\frac{r^3-r^2}{2D}\cdot\dist(z_0,\zero(T))^2,\quad\forall k\geqslant 1.
	\]
	Thus if $D<(r-1)\left(\dfrac{1}{L}+2\rho\right)$, the last-iterate convergence rate of Algorithm \ref{al:sfbsm} is \[
	\dist(0,T(z_k))^2\leqslant\frac{(r-1)^2r^2}{\left[(r-1)\left(\dfrac{1}{L}+2\rho\right)D-D^2\right]k^2}\cdot\dist(z_0,\zero(T))^2, \quad\forall k\geqslant 1.
	\]
	In addition, we have\begin{align*}
		&\ \sum_{k=0}^{\infty}\frac{1}{2}\left[(r-1)\left(\frac{1}{L}+2\rho\right)-D\right](2k+r+1)\norm{F(z_{k+1})+\tilde{G}_{k+1}}^2\\
		\leqslant&\ \frac{r^3-r^2}{2D}\cdot\dist(z_0,\zero(T))^2. 
	\end{align*}
\end{theorem}

\begin{proof}
	Let $\tilde{T}(z_k)=F(z_k)+\tilde{G}_k$. By the definition of $\tilde{G}_k$, we have
	\begin{align*}
		z_{k+1} &=\tilde{z}_{k+1}-\frac{1}{L}[F(z_{k+\frac{1}{2}})+\tilde{G}_{k+1}]-\frac{2\rho k}{k+r}\tilde{T}(z_k).
	\end{align*}
	With above equation, we obtain the following useful equations \begin{align}
		k(z_{k+1}-z_k)&=r(u_k-z_{k+1})-\frac{k+r}{L}[F(z_{k+\frac{1}{2}})+\tilde{G}_{k+1}]-2\rho k\tilde{T}(z_k),\label{eq:usefulthm1_a}\\
		r(z_k-u_k)&=-(k+r)(z_{k+1}-z_k)-\frac{k+r}{L}[F(z_{k+\frac{1}{2}})+\tilde{G}_{k+1}]-2\rho k\tilde{T}(z_k)\label{eq:usefulthm1_b}.
	\end{align}
	
	Also, by $L-$Lipschitz continuity of $F$, we have\[
	\norm{F(z_{k+1})-F(z_{k+\frac{1}{2}})}^2\leqslant L^2\norm{z_{k+1}-z_{k+\frac{1}{2}}}^2.
	\]
	Since \[
	z_{k+1}-z_{k+\frac{1}{2}}=\frac{1}{L}\left[\frac{k}{k+r}\tilde{T}(z_k)-F(z_{k+\frac{1}{2}})-\tilde{G}_{k+1}\right],
	\]
	we have
	\begin{equation}
		\norm{F(z_{k+1})-F(z_{k+\frac{1}{2}})}^2\leqslant\norm{\frac{k}{k+r}\tilde{T}(z_k)-F(z_{k+\frac{1}{2}})-\tilde{G}_{k+1}}^2.
		\label{eq:usefulthm1_c}
	\end{equation}

	Recall the  Lyapunov function
	\[\E(k)=\left[\frac{Dk^2}{2L}+\rho Dk(k-r)\right]\norm{\tilde{T}(z_k)}^2+Drk\inner{\tilde{T}(z_k),z_k-u_k}+\frac{r^3-r^2}{2}\norm{u_k-z^*}^2.
	\]
	First, we divide the difference $\E(k+1)-\E(k)$ into three parts.  \begin{align*}
		&\ \E(k+1)-\E(k)\\
		=&\ \underbrace{\left[\frac{D(k+1)^2}{2L}+\rho D(k+1)(k+1-r)\right]\norm{\tilde{T}(z_{k+1})}^2-\left[\frac{Dk^2}{2L}+\rho Dk(k-r)\right]\norm{\tilde{T}(z_k)}^2}_{\text{I}}\\
		&\ +\underbrace{Dr(k+1)\inner{\tilde{T}(z_{k+1}),z_{k+1}-u_{k+1}}-Drk\inner{\tilde{T}(z_k),z_k-u_k}}_{\text{II}} \\
		&\ +\underbrace{\frac{r^3-r^2}{2}\left[\norm{u_{k+1}-z^*}^2-\norm{u_k-z^*}^2\right]}_{\text{III}}.
	\end{align*}
	
	Secondly, we reckon the upper bound of II and III. First we consider II.
	\begin{align*}
		\text{II}=&\ Dr\inner{\tilde{T}(z_{k+1}),z_{k+1}-u_{k+1}}+Drk\inner{\tilde{T}(z_{k+1}),z_{k+1}-z_k-u_{k+1}+u_k} \\
		&\ +Drk\inner{\tilde{T}(z_{k+1})-\tilde{T}(z_k),z_k-u_k}.
	\end{align*}
	Due to \eqref{eq:usefulthm1_a} and the equality $u_{k+1}-u_k=-\dfrac{D}{r}[F(z_{k+1})+\tilde{G}_{k+1}]$, we have
	\begin{equation}
		\begin{split}
			&\ Drk\inner{\tilde{T}(z_{k+1}),z_{k+1}-z_k-u_{k+1}+u_k}\\
			=&\ Dr\inner{\tilde{T}(z_{k+1}),r(u_k-z_{k+1})-\frac{k+r}{L}[F(z_{k+\frac{1}{2}})+\tilde{G}_{k+1}]-2\rho k \tilde{T}(z_k)}\\
			&\ +D^2k\norm{\tilde{T}(z_{k+1})}^2.
		\end{split}
		\label{eq:thm1_a}
	\end{equation}
	In addition, because of \eqref{eq:usefulthm1_b}, we have\begin{equation}
		\begin{split}
			&\ Drk\inner{\tilde{T}(z_{k+1})-\tilde{T}(z_k),z_k-u_k}\\
			=&\ Dk\inner{\tilde{T}(z_{k+1})-\tilde{T}(z_k),-(k+r)(z_{k+1}-z_k)-\frac{k+r}{L}[F(z_{k+\frac{1}{2}})+\tilde{G}_{k+1}]-2\rho k\tilde{T}(z_k)}.
		\end{split}
		\label{eq:thm1_b}
	\end{equation}
	Also,
	\begin{equation}
		\label{eq:thm1_c}
		Dr\inner{\tilde{T}(z_{k+1}),z_{k+1}-u_{k+1}}=D^2\norm{\tilde{T}(z_{k+1})}^2+Dr\inner{\tilde{T}(z_{k+1}),z_{k+1}-u_k}.
	\end{equation}
	By summing \eqref{eq:thm1_a}, \eqref{eq:thm1_b} and \eqref{eq:thm1_c}, we have
	\begin{align*}
		&\ \text{II}\\
		= &\ -\frac{D(k+r)^2}{L}\inner{F(z_{k+1}), F(z_{k+\frac{1}{2}})}+\frac{D}{L}\inner{F(z_{k+\frac{1}{2}}), k(k+r)\tilde{T}(z_k)-(k+r)^2\tilde{G}_{k+1}} \\
		&\ +\frac{D}{L}\inner{\tilde{G}_{k+1}, k(k+r)\tilde{T}(z_k)-(k+r)^2\tilde{T}(z_{k+1})}\\
		&\ +D^2(k+1)\norm{\tilde{T}(z_{k+1})}^2-Dk(k+r)\inner{\tilde{T}(z_{k+1})-\tilde{T}(z_k), z_{k+1}-z_k}\\
		&\ -2D\rho k(k+r)\inner{\tilde{T}(z_{k+1}),\tilde{T}(z_k)}+2D\rho k^2\norm{\tilde{T}(z_k)}^2\\
		&\ +D(r-r^2)\inner{\tilde{T}(z_{k+1}), z_{k+1}-u_k}.
	\end{align*}
	Combining $\rho-$comonotonicity of $F$, \eqref{eq:usefulthm1_c} and the following equality \[
	-\inner{F(z_{k+1}),F(z_{k+\frac{1}{2}})}=\frac{1}{2}\left[\norm{F(z_{k+1})-F(z_{k+\frac{1}{2}})}^2-\norm{F(z_{k+1})}^2-\norm{F(z_{k+\frac{1}{2}})}^2\right],
	\]
	we obtain the following estimation
	\begin{align*}
		\text{II}\leqslant&-\frac{D(k+r)^2}{2L}\norm{\tilde{T}(z_{k+1})}^2-\rho Dk(k+r)\left(\norm{\tilde{T}(z_{k+1})}^2+\norm{\tilde{T}(z_k)}^2\right)\\
		&+\underbrace{D(r-r^2)\inner{\tilde{T}(z_{k+1}),z_{k+1}-u_k}}_{\text{II}_1}+D^2(k+1)\norm{\tilde{T}(z_{k+1})}^2+\frac{Dk^2}{2L}\norm{\tilde{T}(z_k)}^2.
	\end{align*}
	
	Now estimate III. Owing to the equation
	\begin{equation}
		\frac{1}{2}\norm{a}^2-\frac{1}{2}\norm{b}^2=\inner{a-b,b}+\frac{1}{2}\norm{a-b}^2,
		\label{eq:normdifference}
	\end{equation}
	we have\begin{align*}
		\text{III}&=(r^3-r^2)\inner{u_{k+1}-u_k,u_k-z^*}+\frac{r^3-r^2}{2}\norm{u_{k+1}-u_k}^2\\
		&=-\underbrace{D(r^2-r)\inner{\tilde{T}(z_{k+1}),u_k-z^*}}_{\text{III}_1}+\frac{D^2(r-1)}{2}\norm{\tilde{T}(z_{k+1})}^2.
	\end{align*}
	
	By summing previous estimation, we can deduce the upper bound of $\E(k+1)-\E(k)$. Since \[\text{II}_1-\text{III}_1=-D(r^2-r)\inner{\tilde{T}(z_{k+1}),z_{k+1}-z^*},\]  the upper bound of $\E(k+1)-\E(k)$ can be simplified as follows:\begin{align*}
		&\E(k+1)-\E(k)\\
		\leqslant&\left[-(r-1)\left(\frac{1}{L}+2\rho\right)Dk+D^2k-\frac{Dr^2-D}{2L}+\rho D(1-r)+\frac{D^2(r+1)}{2}\right]\norm{\tilde{T}(z_{k+1})}^2 \\
		&-D(r^2-r)\inner{F(z_{k+1})+\tilde{G}_{k+1},z_{k+1}-z^*}\\
		\leqslant&-\frac{D}{2}\left[(r-1)\left(\frac{1}{L}+2\rho\right)-D\right](2k+r+1)\norm{\tilde{T}(z_{k+1})}^2.
	\end{align*}
	Since $r> 1$, $D\leqslant (r-1)\left(\dfrac{1}{L}+2\rho\right)$, we have\[
	\E(k+1)-\E(k)\leqslant 0.
	\]
	
	Finally, we can obtain the convergence rates of Algorithm \ref{al:sfbsm}. First, we translate $\E(k)$ into the following form.
	\begin{align*}
		\E(k)=&\left[\frac{Dk^2}{2L}+\rho Dk(k-r)-\frac{D^2k^2}{2(r-1)}\right]\norm{\tilde{T}(z_k)}^2+Drk\inner{\tilde{T}(z_k),z_k-z^*}\\
		&+\frac{1}{2}\norm{\frac{Dk}{\sqrt{r-1}}\tilde{T}(z_k)-\sqrt{r^3-r^2}(u_k-z^*)}^2.
	\end{align*}
	By $\rho-$comonotonicity of $T$, we have
	\begin{align*}
		\E(k)\geqslant&\  \left[\left(\frac{1}{2L}+\rho\right)Dk^2-\frac{D^2k^2}{2(r-1)}\right]\norm{\tilde{T}(z_k)}^2\\
		&\ +\frac{1}{2}\norm{\frac{Dk}{\sqrt{r-1}}\tilde{T}(z_k)-\sqrt{r^3-r^2}(u_k-z^*)}^2. 
	\end{align*}
	Since $D\leqslant (r-1)\left(\dfrac{1}{L}+2\rho\right)$, $\left(\dfrac{1}{2L}+\rho\right)Dk^2-\dfrac{D^2k^2}{2(r-1)}\geqslant 0$. Thus $\E(k)\geqslant 0$ and \[
	\left[\left(\frac{1}{2L}+\rho\right)Dk^2-\frac{D^2k^2}{2(r-1)}\right]\norm{\tilde{T}(z_k)}^2\leqslant\E(k)\leqslant\E(0)=\frac{r^3-r^2}{2}\norm{z_0-z^*}^2.
	\] 
	By taking infimum respect to all $z^*\in\zero(T)$, we obtain the desire result. Moreover, by summing $\E(k+1)-\E(k)$ from $0$ to $\infty$ and taking infimum respect to all $z^*\in\zero(T)$, we have
	\begin{align*}
		&\ \sum_{k=0}^\infty\frac{1}{2}\left[(r-1)\left(\frac{1}{L}+2\rho\right)-D\right](2k+r+1)\norm{\tilde{T}(z_{k+1})}^2\\
		\leqslant&\ \frac{r^3-r^2}{2D}\dist(z_0,\zero(T))^2. 
	\end{align*}
\end{proof}

Next, we consider a special monotone inclusion problem:
\begin{equation}
	0\in T(z)=F(z)+N_C(z).
	\label{eq:mainconvexset}
\end{equation}
Here $N_C(z)=\{z^*\in\hh|\inner{z^*,z'-z}\leqslant 0, \forall z'\in C\}$ is the normal cone of convex subset $C$ at $z$ and $F$ is $L-$Lipschitz continuous and monotone. The inclusion problem is equivalent to the Stampacchia Variational Inequality problem\begin{equation}
	\inner{F(z), z'-z}\geqslant 0,\quad\forall z'\in C.
\end{equation}
\begin{example}
	\label{ex:matrixgame}
	Consider the matrix game problem\[
	\min_{x\in\Delta_m}\max_{y\in\Delta_n} x\tr Ay,
	\]
	where $\Delta_m$ is the $m-1$-dimensional unit simplex. Let $C=\Delta_m\times\Delta_n$, the first-order characterization of matrix game is \[
	0\in \binom{Ay}{-A\tr x}+N_C(x,y).
	\]
\end{example}

Based on the projected EG method in \cite{korpelevich76} and the accelerated projected EG method in \cite{cai22}, we propose the \emph{symplectc projected extra-gradient+ (SPEG+) method} , described in Algorithm \ref{al:spegm}.
\begin{algorithm}
	\label{al:spegm}
	\caption{Symplectic Projected Extra-gradient+ Method, SPEG+}
	\textbf{Input: }Operator $F$, Convex set $C$,  $L\in(0,+\infty)$\;
	\textbf{Initialization: }$z_0, u_0=z_0$\;
	\For{$k=0, 1, \cdots$}{
		$\tilde{z}_{k+1} = \dfrac{k}{k+r}z_k+\dfrac{r}{k+r}u_k$\;
		$z_{k+\frac{1}{2}} = P_C\left(\tilde{z}_{k+1}-\dfrac{k}{(k+r)L}F(z_k)\right)$\;
		$z_{k+1} = P_C\left(\tilde{z}_{k+1}-\dfrac{1}{L}F(z_{k+\frac{1}{2}})\right)$\;
		$\tilde{C}_{k+1}  = L(\tilde{z}_{k+1}-z_{k+1})-F(z_{k+\frac{1}{2}})$\;
		$u_{k+1} = u_k-\dfrac{D}{r}[F(z_{k+1})+\tilde{C}_{k+1}]$.
	}	
\end{algorithm}

The corresponding convergence results in given as follows.
\begin{theorem}
	\label{thm:spegrate}
	Let $\{z_k\}$ and $\{u_k\}$ be the sequences generated by Algorithm \ref{al:spegm}. If $F$ is $L-$Lipschitz and monotone, $C$ is closed and convex, $r>1$ and $0<D\leqslant\dfrac{r-1}{L}$, then we have
	\begin{align*}
		\left(\frac{1}{2L}-\frac{D}{2(r-1)}\right)k^2\dist(0, T(z_k))^2&\leqslant\frac{r^3-r^2}{2D}\cdot\dist(z_0,\zero(T))^2,\quad\forall k\geqslant 1,\\
		Drk\inner{F(z_k)+\tilde{C}_k, z_k-z^*}&\leqslant\frac{r^3-r^2}{2}\norm{z_0-z^*}^2,\quad\forall z^*\in\zero(T).
	\end{align*}
	Thus, if $D<\dfrac{r-1}{L}$, the last-iterate convergence rates of Algorithm \ref{al:spegm} are
	\begin{align*}
		\dist(0, T(z_k))^2&\leqslant\cfrac{r^2(r-1)^2}{\left[\cfrac{r-1}{L}D-D^2\right]k^2}\cdot\dist(z_0,\zero(T))^2,\quad\forall k\geqslant 1,\\
		\inner{F(z_k)+\tilde{C}_k, z_k-z^*}&\leqslant\frac{r^2-r}{2Dk}\norm{z_0-z^*}^2,\quad\forall z^*\in\zero(T).
	\end{align*}
	In addition, we have
	\begin{align*}
		& \sum_{k=0}^\infty\frac{1}{2}\left(\frac{r-1}{L}-D\right)(2k+r+1)\norm{F(z_{k+1})+\tilde{C}_{k+1}}^2\leqslant\ \frac{r^3-r^2}{2}\dist(z_0,\zero(T))^2,\\
		& \sum_{k=0}^{\infty}\frac{(k+r)^2}{2L}\norm{\tilde{C}_{k+\frac{1}{2}}-\frac{k}{k+r}\tilde{C}_k}^2\leqslant\frac{r^3-r^2}{2D}\dist(z_0,\zero(T))^2.
	\end{align*}
\end{theorem}

\begin{proof}
	Similar to the proof for Theorem \ref{thm:sfbsmrate} , we need to obtain some useful equalities and inequalities  similar to \eqref{eq:usefulthm1_a}-\eqref{eq:usefulthm1_c}. Here we let $\tilde{C}_{k+\frac{1}{2}}=L\left(\tilde{z}_{k+1}-z_{k+\frac{1}{2}}-\dfrac{k}{(k+r)L}F(z_k)\right)$. By equivalent characterization of projection operator, we have\[
	\tilde{C}_{k+1}\in N_C(z_{k+1}),\quad \tilde{C}_{k+\frac{1}{2}}\in N_C(z_{k+\frac{1}{2}}).
	\]Also by the definition of $\tilde{C}_{k+1}$ and $\tilde{C}_{k+\frac{1}{2}}$, we have\begin{align}
		z_{k+1} & = \tilde{z}_{k+1}-\frac{1}{L}[F(z_{k+\frac{1}{2}})+\tilde{C}_{k+1}],\label{eq:usefulthm2_d}\\
		z_{k+\frac{1}{2}} & = \tilde{z}_{k+1}-\frac{k}{(k+r)L}F(z_k)-\frac{1}{L}\tilde{C}_{k+\frac{1}{2}}\label{eq:usefulthm2_e}.
	\end{align} 
	Using \eqref{eq:usefulthm2_d} and \eqref{eq:usefulthm2_e}, we obtain the following useful equations \begin{align}
		k(z_{k+1}-z_k)&=r(u_k-z_{k+1})-\frac{k+r}{L}[F(z_{k+\frac{1}{2}})+\tilde{C}_{k+1}],\label{eq:usefulthm2_a}\\
		r(z_k-u_k)&=-(k+r)(z_{k+1}-z_k)-\frac{k+r}{L}[F(z_{k+\frac{1}{2}})+\tilde{C}_{k+1}].\label{eq:usefulthm2_b}
	\end{align}
	Also. by $L-$Lipschitz continuity of $F$, we have\[
	\norm{F(z_{k+1})-F(z_{k+\frac{1}{2}})}^2\leqslant L^2\norm{z_{k+1}-z_{k+\frac{1}{2}}}^2.
	\]
	Since $\tilde{C}_{k+\frac{1}{2}}\in N_C(z_{k+\frac{1}{2}}), z_{k+1}\in C$, we have\[
	\inner{\tilde{C}_{k+\frac{1}{2}},z_{k+\frac{1}{2}}-z_{k+1}}\geqslant 0.
	\]
	Thus
	\begin{align*}
		\norm{z_{k+1}-z_{k+\frac{1}{2}}}^2\leqslant&\norm{z_{k+1}-z_{k+\frac{1}{2}}}^2+\frac{2}{L}\inner{\tilde{C}_{k+\frac{1}{2}},z_{k+\frac{1}{2}}-z_{k+1}}\\
		=&\norm{z_{k+1}-z_{k+\frac{1}{2}}-\frac{1}{L}\tilde{C}_{k+\frac{1}{2}}}^2-\frac{1}{L^2}\norm{\tilde{C}_{k+\frac{1}{2}}}^2.
	\end{align*}
	By subtracting \eqref{eq:usefulthm2_d} by \eqref{eq:usefulthm2_e}, we obtain
	\begin{equation}
		\norm{F(z_{k+1})-F(z_{k+\frac{1}{2}})}^2\leqslant\norm{\frac{k}{k+r}F(z_k)-F(z_{k+\frac{1}{2}})-\tilde{C}_{k+1}}^2-\norm{\tilde{C}_{k+\frac{1}{2}}}.
		\label{eq:usefulthm3_c}
	\end{equation}

	Next, we consider the following Lyapunov function\begin{equation}
		\E(k)=\frac{Dk^2}{2L}\norm{F(z_k)+\tilde{C}_k}^2+Drk\inner{F(z_k)+\tilde{C}_k,z_k-u_k}+\frac{r^3-r^2}{2}\norm{u_k-z^*}^2,
	\end{equation}
	where $z^*\in\zero(T)$. First, we divide the difference $\E(k+1)-\E(k)$ into three parts.  \begin{align*}
		&\ \E(k+1)-\E(k)\\
		=&\ \underbrace{\frac{D(k+1)^2}{2L}\norm{F(z_{k+1})+\tilde{C}_{k+1}}^2-\frac{Dk^2}{2L}\norm{F(z_k)+\tilde{C}_k}^2}_{\text{I}}\\
		&\ +\underbrace{Dr(k+1)\inner{F(z_{k+1})+\tilde{C}_{k+1},z_{k+1}-u_{k+1}}-Drk\inner{F(z_k)+\tilde{C}_k,z_k-u_k}}_{\text{II}} \\
		&\ +\underbrace{\frac{r^3-r^2}{2}\left[\norm{u_{k+1}-z^*}^2-\norm{u_k-z^*}^2\right]}_{\text{III}}.
	\end{align*}
	
	Secondly, we reckon the upper bound of II and III. First we consider II. II can be split into the following three terms.
	\begin{align*}
		\text{II}=&\ Dr\inner{F(z_{k+1})+\tilde{C}_{k+1},z_{k+1}-u_{k+1}}\\
		&\ +Drk\inner{F(z_{k+1})+\tilde{C}_{k+1},z_{k+1}-z_k-u_{k+1}+u_k} \\
		&\ +Drk\inner{F(z_{k+1})+\tilde{C}_{k+1}-F(z_k)-\tilde{C}_k,z_k-u_k}.
	\end{align*}
	Due to \eqref{eq:usefulthm2_a} and the equality $u_{k+1}-u_k=-\dfrac{D}{r}[F(z_{k+1})+\tilde{C}_{k+1}]$, we have
	\begin{equation}
		\begin{split}
			&\ Drk\inner{F(z_{k+1}),z_{k+1}-z_k-u_{k+1}+u_k}\\
			=&\ Dr\inner{F(z_{k+1})+\tilde{C}_{k+1},r(u_k-z_{k+1})-\frac{k+r}{L}[F(z_{k+\frac{1}{2}})+\tilde{C}_{k+1}]}\\
			&\ +D^2k\norm{F(z_{k+1})+\tilde{C}_{k+1}}^2.
		\end{split}
		\label{eq:thm2_a}
	\end{equation}
	In addition, because of \eqref{eq:usefulthm2_b}, we have\begin{equation}
		\begin{split}
			&\ Drk\inner{F(z_{k+1})-F(z_k)+\tilde{C}_{k+1}-\tilde{C}_k,z_k-u_k}\\
			=&\ Dk\inner{F(z_{k+1})+\tilde{C}_{k+1}-F(z_k)-\tilde{C}_k,-(k+r)(z_{k+1}-z_k)-\frac{k+r}{L}[F(z_{k+\frac{1}{2}})+\tilde{C}_{k+1}]}.
		\end{split}
		\label{eq:thm2_b}
	\end{equation}
	Also,
	\begin{equation}
		\label{eq:thm2_c}
		\begin{split}
			&\ Dr\inner{F(z_{k+1})+\tilde{C}_{k+1},z_{k+1}-u_{k+1}}\\
			=&\ D^2\norm{F(z_{k+1})+\tilde{C}_{k+1}}^2+Dr\inner{F(z_{k+1})+\tilde{C}_{k+1},z_{k+1}-u_k}.
		\end{split}
	\end{equation}
	By summing \eqref{eq:thm2_a}, \eqref{eq:thm2_b} and \eqref{eq:thm2_c}, we have
	\begin{align*}
		&\ \text{II}\\
		=&\ -\frac{D(k+r)^2}{L}\inner{F(z_{k+1}),F(z_{k+\frac{1}{2}})}+\frac{Dk(k+r)}{L}\inner{F(z_k),F(z_{k+\frac{1}{2}})}-\frac{D(k+r)^2}{L}\norm{\tilde{C}_{k+1}}^2\\
		&\ -\frac{D}{L}\inner{\tilde{C}_{k+1},(k+r)^2[F(z_{k+\frac{1}{2}})+F(z_{k+1})]-k(k+r)[F(z_k)+\tilde{C}_k]}+\frac{Dk(k+r)}{L}\inner{\tilde{C}_k,F(z_{k+\frac{1}{2}})}\\
		&\ -Dk(k+r)\inner{F(z_{k+1})+\tilde{C}_{k+1}-F(z_k)-\tilde{C}_k,z_{k+1}-z_k}\\
		&\ +D(r-r^2)\inner{F(z_{k+1})+\tilde{C}_{k+1},z_{k+1}-u_k}+D^2(k+1)\norm{F(z_{k+1})+\tilde{C}_{k+1}}^2.
	\end{align*}
	Combining $\rho-$comonotonicity of $F$, \eqref{eq:usefulthm3_c} and the following equality \[
	-\inner{F(z_{k+1}),F(z_{k+\frac{1}{2}})}=\frac{1}{2}\left[\norm{F(z_{k+1})-F(z_{k+\frac{1}{2}})}^2-\norm{F(z_{k+1})}^2-\norm{F(z_{k+\frac{1}{2}})}^2\right],
	\]
	we can obtain the upper bound of II as follows:
	\begin{align*}
		\text{II}
		\leqslant&\ -\frac{D(k+r)^2}{2L}\norm{F(z_{k+1})+\tilde{C}_{k+1}}^2+Dk(k+r)\inner{\tilde{C}_k,z_{k+1}-z_k+\frac{1}{L}[F(z_{k+\frac{1}{2}})+\tilde{C}_{k+1}]}\\
		&\ -Dk(k+r)\inner{F(z_{k+1})+\tilde{C}_{k+1}-F(z_k),z_{k+1}-z_k}-\frac{D(k+r)^2}{2L}\norm{\tilde{C}_{k+\frac{1}{2}}}\\
		&\ +D(r-r^2)\inner{F(z_{k+1})+\tilde{C}_{k+1},z_{k+1}-u_k}+D^2(k+1)\norm{F(z_{k+1})+\tilde{C}_{k+1}}^2+\frac{Dk^2}{2L}\norm{F(z_k)}^2.
	\end{align*}
	Since 
	\[z_{k+1}+\frac{1}{L}[F(z_{k+\frac{1}{2}})+\tilde{C}_{k+1}]=z_{k+\frac{1}{2}}+\frac{k}{(k+r)L}F(z_k)+\frac{1}{L}\tilde{C}_{k+1},\]
	we have
	\begin{align*}
		&\ Dk(k+r)\inner{\tilde{C}_k,z_{k+1}-z_k+\frac{1}{L}[F(z_{k+\frac{1}{2}})+\tilde{C}_{k+1}]}\\
		=&\ Dk(k+r)\inner{\tilde{C}_k, z_{k+\frac{1}{2}}-z_k+\frac{1}{L}\left(\frac{k}{k+r}F(z_k)+\tilde{C}_{k+\frac{1}{2}}\right)}.
	\end{align*}
	Thus\begin{align*}
		\text{II}\leqslant & \ -\frac{D(k+r)^2}{2L}\norm{F(z_{k+1})+\tilde{C}_{k+1}}^2+\frac{Dk^2}{2L}\norm{F(z_k)+\tilde{C}_k}^2-\frac{D}{2L}\norm{(k+r)\tilde{C}_{k+\frac{1}{2}}-k\tilde{C}_k}^2 \\
		& \ -Dk(k+r)\inner{F(z_{k+1})+\tilde{C}_{k+1}-F(z_k),z_{k+1}-z_k}+Dk(k+r)\inner{\tilde{C}_k, z_{k+\frac{1}{2}}-z_k}                                                     \\
		& \ +D(r-r^2)\inner{F(z_{k+1})+\tilde{C}_{k+1},z_{k+1}-u_k}+D^2(k+1)\norm{F(z_{k+1})+\tilde{C}_{k+1}}^2.
	\end{align*}
	Now we estimate III. Owing to \eqref{eq:normdifference}, we have\begin{align*}
		\text{III}&=(r^3-r^2)\inner{u_{k+1}-u_k,u_k-z^*}+\frac{r^3-r^2}{2}\norm{u_{k+1}-u_k}^2\\
		&=-\underbrace{D(r^2-r)\inner{F(z_{k+1})+\tilde{C}_{k+1},u_k-z^*}}_{\text{III}_1}+\frac{D^2(r-1)}{2}\norm{F(z_{k+1})+\tilde{C}_{k+1}}^2.
	\end{align*}
	
	Next, we can deduce the upper bound of $\E(k+1)-\E(k)$. Due to \[\text{II}_1-\text{III}_1=-D(r^2-r)\inner{F(z_{k+1})+\tilde{C}_{k+1},z_{k+1}-z^*},\]
	the upper bound of $\E(k+1)-\E(k)$ can be simplified as follows:\begin{align*}
		&\ \E(k+1)-\E(k)\\
		\leqslant&\ -\frac{D}{2}\left(\frac{r-1}{L}-D\right)(2k+r+1)\norm{F(z_{k+1})+\tilde{C}_{k+1}}^2 \\
		&\ -D(r^2-r)\inner{F(z_{k+1})+\tilde{C}_{k+1},z_{k+1}-z^*}-Dk(k+r)\inner{F(z_{k+1})+\tilde{C}_{k+1}-F(z_k),z_{k+1}-z_k}\\
		& \ +Dk(k+r)\inner{\tilde{C}_k, z_{k+\frac{1}{2}}-z_k}-\frac{D}{2L}\norm{(k+r)\tilde{C}_{k+\frac{1}{2}}-k\tilde{C}_k}^2.
	\end{align*}
	Because $\tilde{C}_{k+1}\in N_C(z_{k+1})$, $\tilde{C}_{k}\in N_C(z_k)$, we have\[
	\inner{\tilde{C}_{k+1}, z_k-z_{k+1}}\leqslant 0,\quad\inner{\tilde{C}_k,z_{k+\frac{1}{2}}-z_k}\leqslant 0.
	\]
	In conclusion 
	\begin{align*}
		&\ \E(k+1)-\E(k)\\
		\leqslant&\ -\frac{D}{2}\left(\frac{r-1}{L}-D\right)(2k+r+1)\norm{F(z_{k+1})+\tilde{C}_{k+1}}^2-\frac{D}{2L}\norm{(k+r)\tilde{C}_{k+\frac{1}{2}}-k\tilde{C}_k}^2\\
		\leqslant&\ 0.
	\end{align*}
	
	Finally, we can obtain the convergence results of Algorithm \ref{al:spegm}. First, we translate $\E(k)$ into the following form.
	\begin{align*}
		\E(k)=&\  \left(\frac{D}{2L}-\frac{D^2}{2(r-1)}\right)k^2\norm{F(z_k)+\tilde{C}_k}^2+Drk\inner{F(z_k)+\tilde{C}_k, z_k-z^*}\\
		&\ +\frac{1}{2}\norm{\frac{Dk}{\sqrt{r-1}}[F(z_k)+\tilde{C}_k]-\sqrt{r^3-r^2}(u_k-z^*)}^2.
	\end{align*}
	Since $F$ is monotone, $\E(k)$ can be represented by the sum of three non-negative terms. Then we have
	\begin{align*}
		\left(\frac{D}{2L}-\frac{D^2}{2(r-1)}\right)k^2\dist(0,T(z_k))^2&\leqslant\E(k) \leqslant\E(0)=\frac{r^3-r^2}{2}\norm{z_0-z^*}^2. \\
		Drk\inner{F(z_k)+\tilde{C}_k, z_k-z^*} & \leqslant\E(k) \leqslant\E(0)=\frac{r^3-r^2}{2}\norm{z_0-z^*}^2.
	\end{align*}
	Taking infimum respect to all $z^*\in\zero(T)$ on both sides of the first above inequality and using the inequality $\dist(0,T(z_k))\leqslant\norm{F(z_k)+\tilde{C}_k}$, we obtain the desire result. Moreover, by summing $\E(k+1)-\E(k)$ from $0$ to $\infty$, we have
	\begin{align*}
		&\sum_{k=0}^\infty\frac{D}{2}\left(\frac{r-1}{L}-D\right)(2k+r+1)\norm{F(z_{k+1})+\tilde{C}_{k+1}}^2\leqslant\frac{r^3-r^2}{2}\norm{z_0-z^*}^2,\\
		&\sum_{k=0}^{\infty}\frac{(k+r)^2}{2L}\norm{\tilde{C}_{k+\frac{1}{2}}-\frac{k}{k+r}\tilde{C}_k}^2\leqslant\frac{r^3-r^2}{2}\norm{z_0-z^*}^2.
	\end{align*}
	Also, taking infimum respect to all $z^*\in\zero(T)$, we obtain the desire results.
\end{proof}

\section{Sharper Rate}
\label{sec:4}
In this section, we will build up the $o(1/k^2)$ convergence rate of the SEG type method. It is noteworthy that the symplectic proximal point algorithm admits an $o(1/k^2)$ convergence rate if the assumption $D<r-1$ holds. Analogously, one may suggest that the symplectic EG type method admits $o(1/k^2)$ convergence rate when $D<(r-1)\left(\dfrac{1}{L}+2\rho\right)$. However, the assumption $D<(r-1)\left(\dfrac{1}{L}+2\rho\right)$ alone is not enough. The main reason is that the EG type method tries to use $F(z_{k+\frac{1}{2}})$ to approximate $F(z_{k+1})$, but the error term $\norm{F(z_{k+\frac{1}{2}})-F(z_{k+1})}^2$ can not be estimated directly. To overcome this problem, we need additional assumption on the step-size in the SEG type method. Here, we make use of Algorithm \ref{al:sfbsm} to illustrate the desire result.
\begin{lemma}[Estimation of error]
	\label{lem:error}
	Let $\{z_{\frac{k}{2}}\}$, $\{\tilde{z}_{k+1}\}$ and $\{u_k\}$ be the sequences generated by \begin{align}
		\tilde{z}_{k+1} & = \frac{k}{k+r}z_k+\frac{r}{k+r}u_k,\label{eq:sfbsma} \\
		z_{k+\frac{1}{2}} & = \tilde{z}_{k+1} - \frac{k}{k+r}(s+2\rho) (F(z_k)+\tilde{G}_k),\\
		z_{k+1} &= J_{\frac{1}{s}G}\left[\tilde{z}_{k+1} -sF(z_{k+\frac{1}{2}})-\frac{k}{k+r}2\rho(F(z_k)+\tilde{G}_k)\right],\\
		\tilde{G}_{k+1}&= s^{-1}\left[\tilde{z}_{k+1}-z_{k+1}-sF(z_{k+\frac{1}{2}})-\frac{k}{k+r}2\rho(F(z_k)+\tilde{G}_k)\right],\\
		u_{k+1} &= u_k-\frac{D}{r}[F(z_{k+1})+\tilde{G}_{k+1}].\label{eq:sfbsme}
	\end{align}
	If $F$ is $L-$Lipschitz continuous and $s<\dfrac{1}{L}$, there exists a positive constant $C$ such that\[
	\norm{F(z_{k+1})-F(z_{k+\frac{1}{2}})}^2\leqslant C\norm{F(z_{k+1})+\tilde{G}_{k+1}-F(z_k)-\tilde{G}_k}^2.
	\]
\end{lemma}

\begin{proof}
	First, we show that for all $C_1>1$, there exists a constant $C_2>1$ such that
	\begin{equation}
		\norm{x+y}^2\leqslant C_2\norm{x}^2+C_1\norm{y}^2,\quad\forall x, y\in\hh.
		\label{eq:errorestimation}
	\end{equation}
	By using the Cauchy-Schwarz inequality, we have\begin{align*}
		\norm{x+y}^2 & = \norm{x}^2+\norm{y}^2+2\inner{x,y} \\
		& = \norm{x}^2+\norm{y}^2+2\inner{(C_1-1)^{-\frac{1}{2}}x,\sqrt{C_1-1}y}\\
		& \leqslant \frac{C_1}{C_1-1}\norm{x}^2+C_1\norm{y}^2,\quad\forall C_1>1.
	\end{align*}
	Let $C_2=\dfrac{C_1}{C_1-1}$, we obtain the desire result. Let $\tilde{T}(z_k)=F(z_k)+\tilde{G}_k$. Due to $L-$Lipschitz continuity of $F$, we have\begin{align*}
		\norm{F(z_{k+1})-F(z_{k+\frac{1}{2}})}^2\leqslant L^2s^2\norm{F(z_{k+1})-F(z_{k+\frac{1}{2}})-\tilde{T}(z_{k+1})+\tilde{T}(z_k)}^2.
	\end{align*}
	Let $y=F(z_{k+1})-F(z_{k+\frac{1}{2}})$ and $x=-\tilde{T}(z_{k+1})+\tilde{T}(z_k)$. Since $s<\dfrac{1}{L}$, there exists a constant $C_1$ such that $C_1>1$ and $L^2s^2C_1<1$. By applying \eqref{eq:errorestimation}, we have\[
	\norm{F(z_{k+1})-F(z_{k+\frac{1}{2}})}^2\leqslant L^2s^2 C_2\norm{\tilde{T}(z_{k+1})-\tilde{T}(z_k)}^2+L^2s^2C_1\norm{F(z_{k+1})-F(z_{k+\frac{1}{2}})}^2.
	\]
	Let $C=\dfrac{L^2s^2C_2}{1-L^2s^2C_1}$, we have\[
	\norm{F(z_{k+1})-F(z_{k+\frac{1}{2}})}^2\leqslant C\norm{\tilde{T}(z_{k+1})-\tilde{T}(z_k)}^2.
	\]
\end{proof}

With Lemma \ref{lem:error}, we can demonstrate the $o(1/k^2)$ convergence rate of  the SFBS method.
\begin{theorem}[Faster convergence rate]
	\label{thm:fastsfbsm}
	Let $\{z_{\frac{k}{2}}\}$, $\{\tilde{z}_{k+1}\}$ and $\{u_k\}$ be the sequences generated by \eqref{eq:sfbsma}-\eqref{eq:sfbsme}. If $F$ is $L-$Lipschitz continuous, $G$ is maximally monotone, $\rho-\dfrac{1}{2L}$, $T=F+G$ is $\rho-$comonotone, $\max\{0,-2\rho\}<s<\dfrac{1}{L}$,  and $0<D<(r-1)(s+2\rho)$, we have
	\begin{alignat*}{2}
		&\lim_{k\to\infty}\norm{z_k-u_k}^2=0,\qquad
		&&\lim_{k\to\infty}\norm{u_k-z^*}^2\text{ exists for all }z^*\in\zero(T),\\
		&\lim_{k\to\infty}k^2\dist(0,T(z_k))^2=0,\qquad
		&&\lim_{k\to\infty}k^2\norm{z_{k+1}-z_k}^2=0.
	\end{alignat*}
\end{theorem}
\begin{proof}
	Let $\tilde{T}(z_k)=F(z_k)+\tilde{G}_k$. Owing to $s<\dfrac{1}{L}$, the operator $F$ can be regarded as $s^{-1}-$Lipschitz continuous. By Theorem \ref{thm:sfbsmrate}, we have
	\[\norm{\tilde{T}(z_k)}^2 \leqslant O\left(\frac{1}{k^2}\right),\quad\sum_{k=0}^{\infty}k\norm{\tilde{T}(z_k)}^2 <+\infty.
	\]
	Due to Lemma \ref{lem:error}, we have\[
	\norm{F(z_{k+1})-F(z_{k+\frac{1}{2}})}^2\leqslant O\left(\frac{1}{k^2}\right),\quad\sum_{k=0}^\infty k\norm{F(z_{k+1})-F(z_{k+\frac{1}{2}})}^2<+\infty.
	\]
	Next, we show that \[
	\sum_{k=1}^\infty\frac{\norm{z_{k+1}-u_{k+1}}^2}{2k}<+\infty.
	\]
	To obtain such result, we need to study the following auxiliary Lyapunov function:\[
	\G(k)=\frac{1}{2}\norm{z_k-u_k}^2+\frac{r}{2D}\cdot\frac{(k+r)^2}{k^2}\left(s+2\rho-\frac{D}{r}\right)\norm{u_k-z^*}^2.
	\]
	First, consider the difference $\dfrac{1}{2}\norm{z_{k+1}-u_{k+1}}^2-\dfrac{1}{2}\norm{z_k-u_k}^2$.
	\begin{align*}
		&\ \frac{1}{2}\norm{z_{k+1}-u_{k+1}}^2-\frac{1}{2}\norm{z_k-u_k}^2 \\
		= &\ \frac{1}{2}\inner{z_{k+1}+z_k-u_{k+1}-u_k,z_{k+1}-z_k-u_{k+1}+u_k}\\
		= &\ \frac{1}{2}\left\langle\left(2+\frac{r}{k}\right)z_{k+1}-u_{k+1}-\left(1+\frac{r}{k}\right)u_k+\frac{k+r}{k}s[F(z_{k+\frac{1}{2}})+\tilde{G}_{k+1}]+2\rho \tilde{T}(z_k),\right.\\
		&\qquad\  \left. -\frac{r}{k}z_{k+1}+\left(1+\frac{r}{k}\right)u_k-u_{k+1}-\frac{k+r}{k}s[F(z_{k+\frac{1}{2}})+\tilde{G}_{k+1}]-2\rho \tilde{T}(z_k)\right\rangle\\
		= &\ -\frac{(2k+r)r}{2k^2}\norm{z_{k+1}-u_{k+1}}^2-\frac{(k+r)^2}{k^2}s\inner{F(z_{k+1})-F(z_{k+\frac{1}{2}}), z_{k+1}-u_{k+1}}\\
		&\ -\frac{k+r}{k}2\rho s\inner{F(z_{k+\frac{1}{2}})-F(z_{k+1}), \tilde{T}(z_{k})-\frac{k+r}{k}\tilde{T}(z_{k+1})}\\
		&\ -\frac{(k+r)^2}{k^2}\left(s+2\rho-\frac{D}{r}\right)\inner{\tilde{T}(z_{k+1}),z_{k+1}-u_{k+1}}\\
		&\ -\frac{k+r}{k}\left(s+2\rho-\frac{D}{r}\right)2\rho s\inner{\tilde{T}(z_k)-\frac{k+r}{k}\tilde{T}(z_{k+1}),\tilde{T}(z_{k+1})}\\
		\leqslant &\  -\frac{(2k+r)(r-1)}{2k^2}\norm{z_{k+1}-u_{k+1}}^2+\frac{(k+r)^4s^2}{2k^2(2k+r)}\norm{F(z_{k+1})-F(z_{k+\frac{1}{2}})}^2\\
		&\ -\frac{(k+r)^2}{k^2}\left(s+2\rho-\frac{D}{r}\right)\inner{\tilde{T}(z_{k+1}),z_{k+1}-u_{k+1}}\\
		&\ +\frac{k+r}{k}\rho s\left(\norm{F(z_{k+\frac{1}{2}})-F(z_{k+1})}^2+\norm{\tilde{T}(z_{k})-\frac{k+r}{k}\tilde{T}(z_{k+1})}^2\right)\\
		&\ +\frac{k+r}{k}\left(s+2\rho-\frac{D}{r}\right)\rho s\left(\norm{\tilde{T}(z_{k})-\frac{k+r}{k}\tilde{T}(z_{k+1})}^2+\norm{\tilde{T}(z_{k+1})}^2\right).
	\end{align*}
	Also, since $\dfrac{(k+r)^2}{k^2}$ is non-increasing, we have\begin{align*}
		&\  \frac{r}{2D}\cdot\frac{(k+1+r)^2}{(k+1)^2}\left(s+2\rho-\frac{D}{r}\right)\norm{u_{k+1}-z^*}^2-\frac{r}{2D}\cdot\frac{(k+r)^2}{k^2}\left(s+2\rho-\frac{D}{r}\right)\norm{u_k-z^*}^2 \\
		\leqslant &\  \frac{r}{2D}\cdot\frac{(k+r)^2}{k^2}\left(s+2\rho-\frac{D}{r}\right)\left(\norm{u_{k+1}-z^*}^2-\norm{u_k-z^*}^2\right)\\
		=&\ \frac{r}{2D}\cdot\frac{(k+r)^2}{k^2}\left(s+2\rho-\frac{D}{r}\right)\left(2\inner{u_{k+1}-u_k, u_{k+1}-z^*}-\norm{u_{k+1}-u_k}^2\right)\\
		=&\ \frac{r}{2D}\cdot\frac{(k+r)^2}{k^2}\left(s+2\rho-\frac{D}{r}\right)\left(-\frac{2D}{r}\inner{\tilde{T}(z_{k+1}), u_{k+1}-z^*}-\frac{D^2}{r^2}\norm{\tilde{T}(z_{k+1})}^2\right).
	\end{align*}Let
	\begin{align*}
		\varphi(k)=&\ \frac{(k+r)^4s^2}{2k^2(2k+r)}\norm{F(z_{k+1})-F(z_{k+\frac{1}{2}})}^2\\
		&\ -\frac{(k+r)^2}{k^2}\left(s+2\rho-\frac{D}{r}\right)\left(\rho+\frac{D}{2r}\right)\norm{\tilde{T}(z_{k+1})}^2\\
		&\ +\frac{k+r}{k}\rho s\left(\norm{F(z_{k+\frac{1}{2}})-F(z_{k+1})}^2+\norm{\tilde{T}(z_{k})-\frac{k+r}{k}\tilde{T}(z_{k+1})}^2\right)\\
		&\ +\frac{k+r}{k}\left(s+2\rho-\frac{D}{r}\right)\rho s\left(\norm{\tilde{T}(z_{k})-\frac{k+r}{k}\tilde{T}(z_{k+1})}^2+\norm{\tilde{T}(z_{k+1})}^2\right).
	\end{align*}
	By previous argument, we have\[
	\G(k+1)-\G(k)\leqslant-\frac{(2k+r)(r-1)}{2k^2}\norm{z_{k+1}-u_{k+1}}^2+\varphi(k).
	\]
	Since $\sum\limits_{k=1}^{\infty}\varphi(k)<+\infty$, we have\[
	\G(k+1)+\sum_{i=k+1}^\infty\varphi(i)\leqslant\G(k)+\sum_{i=k}^\infty\varphi(k)-\frac{(2k+r)(r-1)}{2k^2}\norm{z_{k+1}-u_{k+1}}^2.
	\]
	Then we have\[
	\lim_{k\to\infty}\G(k)\text{ exists, }\quad \sum_{k=1}^\infty\frac{\norm{z_k-u_k}^2}{k}\leqslant+\infty.
	\]
	
	The next thing we need to do is to show that for all $z^*\in\zero(T)$, $\lim\limits_{k\to\infty}\norm{u_k-z^*}^2$ exists. If  $\lim\limits_{k\to\infty}\norm{u_k-z^*}^2$  exists, $\lim\limits_{k\to\infty}\G(k)$ exists and $\sum\limits_{k=1}^{\infty}\dfrac{\norm{z_{k+1}-u_{k+1}}^2}{2k}<+\infty$, we can easily show that\[
	\lim_{k\to\infty}\norm{z_k-u_k}^2=0.
	\]
	Calculating the difference of $\dfrac{1}{2}\norm{u_k-z^*}^2$, we have\begin{align*}
		&\ \frac{1}{2}\norm{u_{k+1}-z^*}^2-\frac{1}{2}\norm{u_k-z^*}^2 \\
		\leqslant&\ -\frac{D}{r}\inner{\tilde{T}(z_{k+1}),u_{k+1}-z^*}\\
		\leqslant&\ \frac{D}{2r} \left(k\norm{\tilde{T}(z_{k+1})}^2+k^{-1}\norm{u_{k+1}-z_{k+1}}^2\right)-\frac{D\rho}{r}\norm{\tilde{T}(z_{k+1})}^2.
	\end{align*}
	Since \[
	\sum_{k=1}^\infty\frac{D}{2r} \left(k\norm{T_{k+1}}^2+k^{-1}\norm{u_{k+1}-z_{k+1}}^2\right)-\frac{D\rho}{r}\norm{T_{k+1}}^2<+\infty,
	\]
	we prove the existence of the limitation $\lim\limits_{k\to\infty}\norm{u_k-z^*}^2$. 
	
	Finally, we can prove the results presented in Theorem \ref{thm:sfbsmrate}. Since \[
	\norm{\tilde{T}(z_{k})}^2\leqslant O\left(\frac{1}{k^2}\right),\quad\lim_{k\to\infty}\norm{z_k-u_k}^2=0,
	\]
	we have\[
	\lim_{k\to\infty}Drk\inner{\tilde{T}(z_k),z_k-u_k}=0.
	\]
	Consider the following Lyapunov function
	\begin{align*}
		\E(k)=&\ \left[\frac{Dk^2s}{2}+\rho Dk(k-r)\right]\norm{\tilde{T}(z_k)}^2\\
		&\ +Drk\inner{\tilde{T}(z_k),z_k-u_k}+\frac{r^3-r^2}{2}\norm{u_k-z^*}. 
	\end{align*}
	By Theorem \ref{thm:fastsfbsm}, we have $\{\E(k)\}$ is non-increasing and non-negetive. Thus, $\lim\limits_{k\to\infty}\E(k)$ exists. Also, by the existence of  $\lim\limits_{k\to\infty}\norm{u_k-z^*}^2$, we have\[
	\lim_{k\to\infty}\left[\frac{Dk^2s}{2}+\rho Dk(k-r)\right]\norm{\tilde{T}(z_k)}^2\text{ exists}.
	\]
	Since $\sum\limits_{k=0}^\infty k\norm{\tilde{T}(z_k)}^2<+\infty$, we have\[
	\lim_{k\to\infty}k^2\norm{\tilde{T}(z_k)}^2=0.
	\]
	By the inequality $0\leqslant\dist(0,T(z_k))^2\leqslant\norm{\tilde{T}(z_k)}^2$, we can obtain\[
	\lim_{k\to\infty}k^2\dist(0,T(z_k))^2=0.
	\]
	Also, by the inequality\[
	k(z_{k+1}-z_k)=r(u_{k+1}-z_{k+1})-[(k+r)s-D]\tilde{T}(z_{k+1})-(k+r)s[F(z_{k+\frac{1}{2}})-F(z_{k+1})]-2\rho k \tilde{T}(z_k),
	\]
	we have\[
	\lim_{k\to\infty}k^2\norm{z_{k+1}-z_k}^2=0.
	\]
\end{proof}

 Next, we discuss the weak convergence property.
\begin{lemma}
	\label{lem:maximallycomonotone}
	Let $F:\mathcal{H}\to\mathcal{H}$ be a $L$-Lipschitz continuous operator, $G$ be a maximally monotone operator, and $T=F+G$ be a  $\rho$-comonotone operator. If $\rho>-\dfrac{1}{2L}$, then $\frac{1}{2L}(F+G)$ is a maximally $2L\rho$-comonotone.
\end{lemma}
\begin{proof}
	Because of $L$-Lipschitz continuity of $F$, we have
	\[
	\inner{F(x)-F(y),x-y}\geq-\norm{F(x)-F(y)}_2\norm{x-y}_2\geq-L\norm{x-y}_2^2,\quad\forall x,y\in\mathcal{H}.
	\]  
	Then $LI+F$ is a monotone operator. Also, by Corollary 20.28 in \cite{bauschke17}, $LI+F$ is a maximally monotone operator. Then according to Minty's surjective theorem and Theorem 25.2 in \cite{bauschke17}, we have  
	\[
	\mathrm{ran}\left(2LI+F+G\right)=\mathrm{ran}\left(LI+\left(LI+F+G\right)\right)=\mathcal{H}.
	\] 
	Because $\frac{1}{2L}(F+G)$ is $2L\rho$-comonotone and $2L\rho>-1$, then by Proposition 2.16 in \cite{bauschke21}, we have that $\frac{1}{2L}(F+G)$ is maximally $2L\rho$-comonotone.  
\end{proof}
\begin{lemma}  
	\label{lem:comonotoneconverge}  
	Let $T$ be a maximally $\rho$-comonotone operator. If there exists two sequences $\{z_k\}$ and $\{u_k\}$ such that $u_k\in T(z_k)$, $u_k\to u$ and $z_k\rightharpoonup z$, then$z\in T(x)$.  
\end{lemma}  
Above lemma is a direct generalization of Proposition 20.38 in \cite{bauschke17}.
\begin{proof}  
	Let $(y,v)$ be any point such that $v\in T(y)$. Then we have  
	\[
	\inner{u_k-v,z_k-y}\geq\rho\norm{u_k-v}_2^2.
	\]  
	Let $k\to\infty$, we have
	\[
	\inner{u-v,z-y}\geq\rho\norm{u-v}_2^2,\quad\forall v\in T(x).
	\]  
	Because $T$ is a maximally $\rho$-comonotone operator, then $u\in T(x)$. If $u\notin T(z)$, then consider the following operator
	\[
	\overline{T}(y):=\begin{cases}  
		T(y),& y\neq z;\\  
		T(z)\cup\{u\},& y=z.  
	\end{cases}
	\]  
	Then it is easy to verify that $\overline{T}$ is $\rho$-comonotone and the graph of $T$ is a proper subset of the graph of $\tilde{T}$, which is contradict to the assumption that $T$ be a maximally $\rho$-comonotone operator.  
\end{proof}  

\begin{theorem}[Weak convergence property]  
	\label{thm:weakconvergence}  
	Let $\{z_{\frac{k}{2}}\}$, $\{\tilde{z}_{k+1}\}$ and $\{u_k\}$ be the sequences generated by \eqref{eq:sfbsma}-\eqref{eq:sfbsme}. If $F$ is $L-$Lipschitz continuous, $G$ is maximally monotone, $T=F+G$ is $\rho-$comonotone, $\rho>-\dfrac{1}{2L}$,  $\max\{0,-2\rho\}<s<\dfrac{1}{L}$,  and $0<D<(r-1)(s+2\rho)$, then both of the sequences $\{z_k\}$ and $\{u_k\}$ converge weakly to the same point in $\zero(T)$.  
\end{theorem}  

\begin{proof}  
	According to Theorem \ref{thm:fastsfbsm} and Proposition \ref{prop:discreteopial}, we only need to prove that every weak cluster points of $\{u_k\}$ belong to $\zero(T)$. Let $u_\infty$ be any weak cluster point of $\{u_k\}$, and $\{u_{k_j}\}$ be the subsequence of $\{u_k\}$ such that$u_{k_j}\rightharpoonup u_\infty$.  
	
	Then by Theorem \ref{thm:fastsfbsm}, we have $z_{k_j}\rightharpoonup u_\infty$. Because $\tilde{T}(z_{k_j})\in T(z_{k_j})$ and $\tilde{T}(z_{k_j})\to 0$, we have $\frac{1}{2L}\tilde{T}(z_{k_j})\to0$. By Lemma \ref{lem:maximallycomonotone} and \ref{lem:comonotoneconverge}, we have $0\in T(u_\infty)$, i. e. $u_\infty\in\zero(T)$.  
\end{proof}  

The method to demonstrate the $o(1/k^2)$ convergence rate and the weak convergence property of the SPEG+ method is similar to the proof for Theorem \ref{thm:fastsfbsm} and Theorem \ref{thm:weakconvergence}.
\begin{corollary}
	\label{coro:fastpseg+}
	Let $\{z_{\frac{k}{2}}\}$, $\{\tilde{z}_{k+1}\}$ and $\{u_k\}$ be the sequences generated by\begin{align*}
		\tilde{z}_{k+1} &= \dfrac{k}{k+r}z_k+\dfrac{r}{k+r}u_k,\\
		z_{k+\frac{1}{2}} &= P_C\left(\tilde{z}_{k+1}-\dfrac{ks}{(k+r)}F(z_k)\right),\\
		z_{k+1} &= P_C\left(\tilde{z}_{k+1}-sF(z_{k+\frac{1}{2}})\right),\\
		\tilde{C}_{k+1} & = s^{-1}\left(\tilde{z}_{k+1}-sF(z_{k+\frac{1}{2}})-z_{k+1}\right),\\
		u_{k+1} &= u_k-\dfrac{D}{r}[F(z_{k+1})+\tilde{C}_{k+1}].
	\end{align*} If $F$ is $L-$Lipschitz continuous and monotone, $0<s<\dfrac{1}{L}$, $r>1$ and $0<D<(r-1)s$, then we have\begin{align*}
		\lim_{k\to\infty}\norm{z_k-u_k}^2&=0,\\
		\lim_{k\to\infty}\norm{u_k-z^*}^2&\text{ exists for all }z^*\in\zero(T),\\
		\lim_{k\to\infty}k^2\dist(0,T(z_k))^2&=0,\\
		\lim_{k\to\infty}k\inner{F(z_k)+\tilde{C}_k, z_k-z^*}&=0,\quad\forall z^*\in\zero(T),\\
		\lim_{k\to\infty}k^2\norm{z_{k+1}-z_k}^2&=0.
	\end{align*}
	Moreover, both of the sequences $\{z_k\}$ and $\{u_k\}$ converge weakly to a point in $\zero(T)$.
\end{corollary}

\begin{proof}
	The main difficultly is to estimate the error term $\norm{F(z_{k+1})-F(z_{k+\frac{1}{2}})}$. If\[
	\sum_{k=0}^{\infty}k\norm{F(z_{k+1})-F(z_{k+\frac{1}{2}})}^2<+\infty
	\]
	is true, then by applying the proofing technique in the proof of Theorem \ref{thm:fastsfbsm}, we can obtain the results in Corollary \ref{coro:fastpseg+}. Since $F$ is $L-$Lipschitz continuous, we have
	\[\norm{F(z_{k+1})-F(z_{k+\frac{1}{2}})}^2 \leqslant L^2s^2\norm{F(z_{k+\frac{1}{2}})+\tilde{C}_{k+1}-\frac{k}{k+r}F(z_k)-\tilde{C}_{k+\frac{1}{2}}}^2.\]
	Since $s<\dfrac{1}{L}$, there exists a constant $C_1>1$ such that $L^2s^2C_1<1$. By \eqref{eq:errorestimation}, there exists a constant $C_2$ such that
	\begin{align*}
		&\ \norm{F(z_{k+1})-F(z_{k+\frac{1}{2}})}^2\\
		\leqslant&\  L^2s^2C_1\norm{F(z_{k+1})-F(z_{k+\frac{1}{2}})}^2\\ 
		&\ +L^2s^2C_2\norm{F(z_{k+1})+\tilde{C}_{k+1}-\frac{k}{k+r}F(z_k)-\tilde{C}_{k+\frac{1}{2}}}^2\\
		\leqslant&\ L^2s^2C_1\norm{F(z_{k+1})-F(z_{k+\frac{1}{2}})}^2+2L^2s^2C_2\norm{\frac{k}{k+r}\tilde{C}_k-\tilde{C}_{k+1}}^2\\
		&\ +2L^2s^2C_2\norm{F(z_{k+1})+\tilde{C}_{k+1}-\frac{k}{k+r}F(z_k)-\frac{k}{k+r}\tilde{C}_k}^2.
	\end{align*}
	The proof of Theorem \ref{thm:spegrate} shows that\[
	\sum_{k=0}^\infty(k+r)^2\norm{\frac{k}{k+r}\tilde{C}_k-\tilde{C}_{k+\frac{1}{2}}}^2<+\infty.
	\]
	In conclusion,\[
	\sum_{k=0}^{\infty}k\norm{F(z_{k+1})-F(z_{k+\frac{1}{2}})}^2<+\infty
	\]holds.
\end{proof}

\section{Line Search Framework}
\label{sec:5}
The theorems presented in Section \ref{sec:3} indicate that the SEG+ type method exhibits its maximum theoretical convergence speed when $D=(r-1)\left(\dfrac{1}{2L}+\rho\right)$. This prompts the question as to whether the SEG+ type method with $D=(r-1)\left(\dfrac{1}{2L}+\rho\right)$ indeed converges most rapidly in practical applications. We conduct a straightforward numerical experiment aimed at verifying this question. Here we consider the following min-max problem:\[
\min_x\max_y f(x,y)=-\frac{1}{6}x^2+\frac{2\sqrt{2}}{3}xy+\frac{1}{6}y^2.
\]
Such min-max problem is considered in section 4.4 in \cite{lee21}. The first-order criteria of above min-max problem is 
\begin{equation}
	0=F(x,y)=\begin{pmatrix}
		-\dfrac{1}{3}x+\dfrac{2\sqrt{2}}{3}y\\
		-\dfrac{2\sqrt{2}}{3}x-\dfrac{1}{3}y
	\end{pmatrix}.
	\label{eq:seg+linesearchoriginal}
\end{equation}
It is easily to verify that \begin{align*}
	\norm{F(x,y)-F(x',y')} & =\norm{(x,y)-(x',y')} ,\quad\forall (x,y),\  (x',y'), \\
	\inner{F(x,y)-F(x',y'),x-y} & =-\frac{1}{3}\norm{F(x,y)-F(x',y')}^2,\quad\forall (x,y),\  (x',y').
\end{align*}

We apply Algorithm \ref{al:sfbsm} with parameter $r = 2$ and different parameters $D$ to above zero-point problem until the termination condition is satisfied. The required number of iterations of Algorithm \ref{al:sfbsm} respect to $D$  is plotted in Figure \ref{fig:seg+d}.
\begin{figure}[ht]
	\centering
	\includegraphics[width=.6\textwidth]{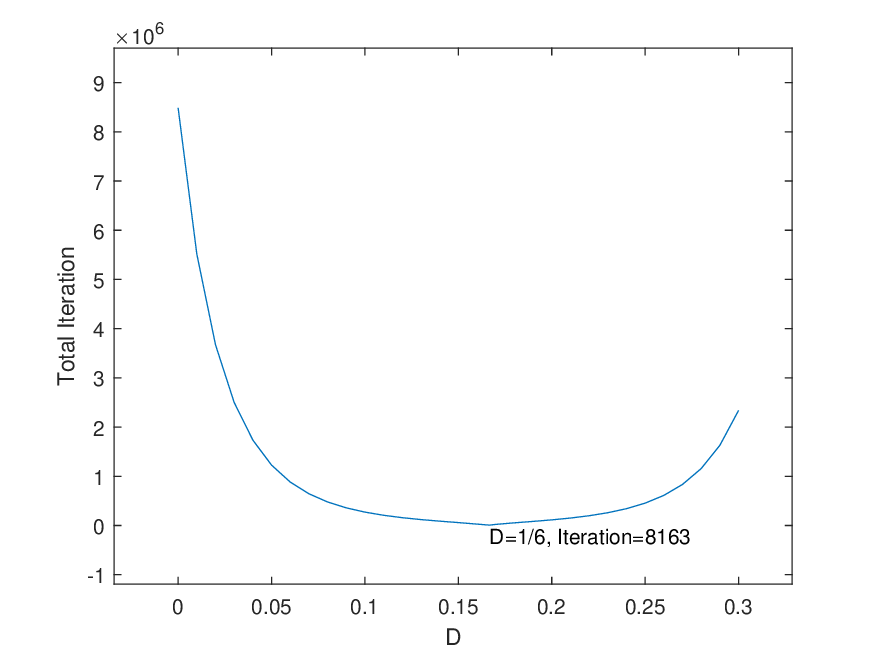}
	\caption{The required number of iterations Algorithm \ref{al:sfbsm} with $G=0$ respect to $D$. The termination condition is $\norm{F(x_k,y_k)}\leqslant10^{-6}$.}
	\label{fig:seg+d}
\end{figure}

As we can see from Figure \ref{fig:seg+d}, when $D=(r-1)\left(\dfrac{1}{2L}+\rho\right)$, i. e. $D=\dfrac{1}{6}$, the corresponding instance of Algorithm \ref{al:sfbsm} requires least number of iterations. However, the Lipschitz constant $L$  and the comonotone index $\rho$ of the operator $F$ described in \eqref{eq:seg+linesearchoriginal} can be deduced in advance. In practice, the Lipschitz constant of most of the operators $F$ and the comonotone index of most of the operators $T$ can not be determined. In addition, even though we can determine the Lipschitz constant and the comonotone index in advance, the operator $F$ may have smaller local Lipschitz constant or the operator $T$ may have larger local comonotone index as $z_k$ approaching $\zero(T)$. Algorithm \ref{al:sfbsm} can not catch these local information. To overcome these problems, we apply the line search framework to the SFBS method to estimate these local information. 

Since there is a term $\inner{F(z_k)+\tilde{G}_k, z_k-u_k}$ in the Lyapunov function, when we consider the difference of $\inner{F(z_k)+\tilde{G}_k, z_k-u_k}$, $\inner{F(z_{k+1})+\tilde{G}_{k+1}-F(z_k)-\tilde{G}_k, z_{k+1}-z_k}$ appears. This is why we need to assume comonotonicity for the operator $T=F+G$. Thus we propose the termination condition for $\rho_k$ is\[
\rho_k\leqslant\frac{\inner{F(z_{k+1})+\tilde{G}_{k+1}-F(z_k)+\tilde{G}_k, z_{k+1}-z_k}}{\norm{F(z_{k+1})+\tilde{G}_{k+1}-F(z_k)-\tilde{G}_k}^2}.
\]
Also, the idea that $F(z_{k+\frac{1}{2}})$ can be seen as an approximation of $F(z_{k+1})$ inspires us that the Lipschitz continuity on $F$ helps us give an estimation of the error $\norm{F(z_{k+\frac{1}{2}})-F(z_{k+1})}$. Because the parameter $L_k$ serves as an estimation of local Lipschitz constant, the termination condition for $L_k$ should be\[
L_k\geqslant\frac{\norm{F(z_{k+\frac{1}{2}})-F(z_{k+1})}}{\norm{z_{k+\frac{1}{2}}-z_{k+1}}}.
\]

By summarizing previous argument, we obtain Algorithm \ref{al:linesearchseg+}. The convergence results of Algorithm \ref{al:linesearchseg+} is given in the following theorems.
\begin{algorithm}[ht]
	\label{al:linesearchseg+}
	\caption{Line Search Framework for SEG+}
	\textbf{Input: } Operator $F$ and Operator $G$\;
	\textbf{Initialize: } $z_0, u_0=z_0$\;
	\For{$k=0, 1, \cdots$}{
		\Repeat{$\inner{F(z_{k+1})+\tilde{G}_{k+1}-F(z_k)-\tilde{G}_k,z_{k+1}-z_k}\geqslant\rho_k\norm{F(z_{k+1})+\tilde{G}_{k+1}-F(z_k)-\tilde{G}_k}^2$, $\norm{F(z_{k+1})-F(z_{k+\frac{1}{2}})}\leqslant L_k\norm{z_{k+1}-z_{k+\frac{1}{2}}}$}{
			Choosing $\alpha_k$, $L_k$, $\rho_k$ with $\rho_k>-\dfrac{1}{2L_k}$\;
			$\tilde{z}_{k+1}=\alpha_k u_k+(1-\alpha_k)z_k$\;
			$z_{k+\frac{1}{2}}=\tilde{z}_{k+1}-(1-\alpha_k)\left(\dfrac{1}{L_k}+2\rho_k\right)(F(z_k)+\tilde{G}_k)$\;
			$z_{k+1}=J_{L_k G}\left(\tilde{z}_{k+1}-\dfrac{1}{L_k}F(z_{k+\frac{1}{2}})-(1-\alpha_k)2\rho_k(F(z_k)+\tilde{G}_k)\right)$\;
			$\tilde{G}_{k+1}=L_k\left[\tilde{z}_{k+1}-z_{k+1}-\dfrac{1}{L_k}F(z_{k+\frac{1}{2}})-(1-\alpha_k)2\rho_k(F(z_k)+\tilde{G}_k)\right]$\;
		}
		$u_{k+1}=u_k-\dfrac{D}{r}\left(\dfrac{1}{2L_k}+\rho_k\right)(F(z_{k+1})+\tilde{G}_{k+1})$.
	}
\end{algorithm}

\begin{theorem}
	\label{thm:linesearchseg+}
	Let $\{z_{\frac{k}{2}}\}$, $\{\tilde{z}_{k+1}\}$ and $\{u_k\}$ be the sequences generated by Algorithm \ref{al:linesearchseg+}, and let $\{\E(k)\}$ be the Lyapunov function such that\begin{align*}
		\E(k) =&\ D\left[\left(\sum_{i=0}^{k-1}\frac{1}{2L_i}+\rho_i\right)^2-r\rho_{k-1}\left(\sum_{i=0}^{k-1}\frac{1}{2L_i}+\rho_i\right)\right]\norm{F(z_k)+\tilde{G}_k}^2\\
		&\ +Dr\left(\sum_{i=0}^{k-1}\frac{1}{2L_i}+\rho_i\right)\inner{F(z_k)+\tilde{G}_k,z_k-u_k}+\frac{r^3-r^2}{2}\norm{u_k-z^*}^2,
	\end{align*} 
	where $z^*\in\hh$.	If $r>1$, $0<D<2(r-1)$,\[
	\begin{gathered}
		\alpha_k= \frac{r\left(\cfrac{1}{2L_k}+\rho_k\right)}{\left(\sum_{i=0}^{k-1}\cfrac{1}{2L_i}+\rho_i\right)+r\left(\cfrac{1}{2L_k}+\rho_k\right)},\\
		\inner{F(z_{k+1})+\tilde{G}_{k+1}-F(z_k)-\tilde{G}_k,z_{k+1}-z_k} \geqslant\rho_k\norm{F(z_{k+1})+\tilde{G}_{k+1}-F(z_k)-\tilde{G}_k}^2,\\
		\norm{F(z_{k+1})-F(z_{k+\frac{1}{2}})}\leqslant L_k\norm{z_{k+1}-z_{k+\frac{1}{2}}},
	\end{gathered}
	\]
	$\rho_k>-\dfrac{1}{2L_k}$ and $\{\rho_k\}$ is non-decreasing, then the upper bound of the difference $\E(k+1)-\E(k)$ is
	\begin{align*}
		&\ \E(k+1)-\E(k)\\
		\leqslant&\ -\frac{D}{2}\left(\frac{1}{2L_k}+\rho_k\right)\left(r-1-\frac{D}{2}\right)\left[(r-1)\left(\frac{1}{2L_k}+\rho_k\right)+2\Sigma_{k+1}\right]\norm{F(z_{k+1})+\tilde{G}_{k+1}}^2\\
		&\ +D(r^2-r)\left(\frac{1}{2L_k}+\rho_k\right)\left(\rho_k\norm{\tilde{T}(z_{k+1})}^2-\inner{\tilde{T}(z_{k+1}),z_{k+1}-z^*}\right).
	\end{align*}
\end{theorem}

\begin{proof}
	For the sake of simplicity, let $\Sigma_k=\sum\limits_{i=0}^{k-1}\dfrac{1}{2L_i}+\rho_i$ and $\tilde{T}(z_k)=F(z_k)+\tilde{G}_k$. First, we have
	\begin{align}
		\Sigma_k(z_{k+1}-z_k)  = &\  r\left(\frac{1}{2L_k}+\rho_k\right)(u_k-z_{k+1})-2\rho_k\Sigma_k \tilde{T}(z_k)\notag\\
		&\ -\frac{1}{L_k}\cdot\left[r\left(\frac{1}{2L_k}+\rho_k\right)+\Sigma_k\right](F(z_{k+\frac{1}{2}})+\tilde{G}_{k+1}),\label{proof:linesearchseg+a}\\	 	
		r\left(\frac{1}{2L_k}+\rho_k\right)(z_k-u_k)= &\ -\left[r\left(\frac{1}{2L_k}+\rho_k\right)+\Sigma_k\right](z_{k+1}-z_k)-2\rho_k\Sigma_k(F(z_k)+\tilde{G}_k)\notag\\
		&\ -\frac{1}{L_k}\left[r\left(\frac{1}{2L_k}+\rho_k\right)+\Sigma_k\right](F(z_{k+\frac{1}{2}})+\tilde{G}_{k+1}).\label{proof:linesearchseg+b}
	\end{align}

	Next, we divide the difference $\E(k+1)-\E(k)$ into three parts.\begin{align*}
		\E(k+1)-\E(k)= &\ \underbrace{D(\Sigma_{k+1}^2-r\rho_k\Sigma_{k+1})\norm{\tilde{T}(z_{k+1})}^2-D(\Sigma_k^2-r\rho_{k-1}\Sigma_k)\norm{\tilde{T}(z_k)}^2}_{\text{I}} \\
		&\ +\underbrace{Dr\Sigma_{k+1}\inner{\tilde{T}(z_{k+1}),z_{k+1}-u_{k+1}}-Dr\Sigma_k\inner{\tilde{T}(z_k),z_k-u_k}}_{\text{II}}\\
		&\ +\underbrace{\frac{r^3-r^2}{2}\left[\norm{u_{k+1}-z^*}^2-\norm{u_k-z^*}^2\right]}_{\text{III}}.
	\end{align*}
	The upper bound of II and III can be estimated as follows. First, we consider II. We split II into three parts as follows:\begin{align*}
		\text{II}= &\ Dr\left(\frac{1}{2L_k}+\rho_k\right)\inner{\tilde{T}(z_{k+1}),z_{k+1}-u_{k+1}}+Dr\Sigma_k\inner{\tilde{T}(z_{k+1}),z_{k+1}-z_k-u_{k+1}+u_k} \\
		&\ +Dr\Sigma_k\inner{\tilde{T}(z_{k+1})-\tilde{T}(z_k),z_k-u_k}.
	\end{align*}
	By \eqref{proof:linesearchseg+a}, \eqref{proof:linesearchseg+b} and $u_{k+1}=u_k-\dfrac{D}{r}\left(\dfrac{1}{2L_k}+\rho_k\right)\tilde{T}(z_{k+1})$, we have\begin{align*}
		\text{II}= &\ Dr\left(\frac{1}{2L_k}+\rho_k\right)\inner{\tilde{T}(z_{k+1}),z_{k+1}-u_k}+D^2\left(\frac{1}{2L_k}+\rho_k\right)\Sigma_{k+1}\norm{\tilde{T}(z_{k+1})}^2 \\
		&\ +Dr\inner{\tilde{T}(z_{k+1}),r\left(\frac{1}{2L_k}+\rho_k\right)(u_k-z_{k+1})-2\rho_k\Sigma_k \tilde{T}(z_k)} \\
		&\ -Dr\inner{\tilde{T}(z_{k+1}),\frac{1}{L_k}\left[r\left(\frac{1}{2L_k}+\rho_k\right)+\Sigma_k\right](F(z_{k+\frac{1}{2}})+\tilde{G}_{k+1})}\\
		&\ -D\left(\frac{1}{2L_k}+\rho_k\right)^{-1}\Sigma_k\left[r\left(\frac{1}{2L_k}+\rho_k\right)+\Sigma_k\right]\inner{\tilde{T}(z_{k+1})-\tilde{T}(z_k),z_{k+1}-z_k}\\
		&\ -D\left(\frac{1}{2L_k}+\rho_k\right)^{-1}\Sigma_k\inner{\tilde{T}(z_{k+1})-\tilde{T}(z_k),\frac{1}{L_k}\left[r\left(\frac{1}{2L_k}+\rho_k\right)+\Sigma_k\right](F(z_{k+\frac{1}{2}})+\tilde{G}_{k+1})}\\
		&\ -D\left(\frac{1}{2L_k}+\rho_k\right)^{-1}\Sigma_k\inner{\tilde{T}(z_{k+1})-\tilde{T}(z_k),2\rho_k\Sigma_k\tilde{T}(z_k)}.
	\end{align*}
	By the assumption \[
	\norm{F(z_{k+1})-F(z_{k+\frac{1}{2}})}\leqslant L_k\norm{z_{k+1}-z_{k+\frac{1}{2}}},
	\]
	we have
	\[\norm{F(z_{k+1})-F(z_{k+\frac{1}{2}})}^2\leqslant \norm{F(z_{k+\frac{1}{2}})+\tilde{G}_{k+1}-\left(\Sigma_k+\frac{r}{2L_k}+r\rho_k\right)^{-1}\Sigma_k\tilde{T}(z_k)}^2.\] 
	Combining previous estimation and the assumption\[
	\inner{\tilde{T}(z_{k+1})-\tilde{T}(z_k),z_{k+1}-z_k}\geqslant\rho_k\norm{\tilde{T}(z_{k+1})-\tilde{T}(z_k)}^2,
	\]
	we can deduce the upper bound of II as follows:\begin{align*}
		\text{II}\leqslant &  \ -D\cfrac{\left[r\left(\cfrac{1}{2L_k}+\rho_k\right)+\Sigma_k\right]^2}{1+2L_k\rho_k}\norm{\tilde{T}(z_{k+1})}^2+D^2\left(\frac{1}{2L_k}+\rho_k\right)\Sigma_{k+1}\norm{\tilde{T}(z_{k+1})}^2\\
		&\ -D\rho_k\left(\frac{1}{2L_k}+\rho_k\right)^{-1}\Sigma_k\left[r\left(\frac{1}{2L_k}+\rho_k\right)+\Sigma_k\right]\norm{\tilde{T}(z_{k+1})}^2\\
		&\ +D(\Sigma_k^2-r\rho_k\Sigma_k)\norm{\tilde{T}(z_k)}^2+D(r^2-r)\left(\frac{1}{2L_k}+\rho_k\right)\inner{\tilde{T}(z_{k+1}),u_k-z_{k+1}}.
	\end{align*}
	Next, we consider III. By \eqref{eq:normdifference}, we have
	\[\text{III}= \ -D(r^2-r)\left(\frac{1}{2L_k}+\rho_k\right)\inner{\tilde{T}(z_{k+1}),u_k-z^*}+\frac{D(r-1)}{2}\left(\frac{1}{2L_k}+\rho_k\right)^2\norm{\tilde{T}(z_{k+1})}^2.\]
	
	Finally, we can estimate the upper bound of $\E(k+1)-\E(k)$. Combining previous estimations and non-decreasing of $\{\rho_k\}$, we have\begin{align*}
		&\ \E(k+1)-\E(k)\\
		\leqslant &\ -D\left[\frac{(r-1)^2}{2L_k}\left(\frac{1}{2L_k}+\rho_k\right)+\Sigma_{k+1}^2+\frac{r-1}{L_k}\Sigma_{k+1}-(r-1)\left(\frac{1}{2L_k}+\rho_k\right)\rho_k\right]\norm{\tilde{T}(z_{k+1})}^2\\
		&\ +D\left[\Sigma_{k+1}^2-2(r-1)\rho_k\Sigma_{k+1}+D\left(\frac{1}{2L_k}+\rho_k\right)\Sigma_{k+1}+\frac{D(r-1)}{2}\left(\frac{1}{2L_k}+\rho_k\right)^2\right]\norm{\tilde{T}(z_{k+1})}^2\\
		&\ -D(r^2-r)\left(\frac{1}{2L_k}+\rho_k\right)\inner{\tilde{T}(z_{k+1}),z_{k+1}-z^*}.
	\end{align*}
	By adding and subtracting $D(r^2-r)\rho_k\left(\dfrac{1}{2L_k}+\rho_k\right)\norm{\tilde{T}(z_{k+1})}^2$, we have
	\begin{align*}
		\E(k+1)-\E(k)\leqslant&\ -\frac{D}{2}\left(\frac{1}{2L_k}+\rho_k\right)\left(r-1-\frac{D}{2}\right)\left[(r-1)\left(\frac{1}{2L_k}+\rho_k\right)+2\Sigma_{k+1}\right]\norm{\tilde{T}(z_{k+1})}^2\\
		&\ +D(r^2-r)\left(\frac{1}{2L_k}+\rho_k\right)\left(\rho_k\norm{\tilde{T}(z_{k+1})}^2-\inner{\tilde{T}(z_{k+1}),z_{k+1}-z^*}\right).
	\end{align*}
	At this point, we obtain the desired result.
\end{proof}
By using Theorem \ref{thm:linesearchseg+}, one can easily propose a sufficient condition to guarantee the convergence of Algorithm \ref{al:linesearchseg+}.
\begin{corollary}
	\label{coro:linesearchseg+}
	Let $\{z_{\frac{k}{2}}\}$, $\{\tilde{z}_{k+1}\}$ and $\{u_k\}$ be the sequences generated by Algorithm \ref{al:linesearchseg+}. If the parameters $\{L_k\}$, $\{\rho_k\}$ and $D$ satisfy the assumptions in Theorem \ref{thm:linesearchseg+} and there exists a $z^*$ such that the inequality\[
	\inner{F(z_{k+1}+\tilde{G}_{k+1}, z_{k+1}-z^*}\geqslant\rho_k\norm{F(z_{k+1})+\tilde{G}_{k+1}}^2
	\]
	is hold for all $k=0$, $1$, $\cdots$, then we have\[
	\dist(0,T(z_k))^2\leqslant\cfrac{r^2(r-1)^2}{[2(r-1)D-D^2]\left(\sum\limits_{i=0}^{k-1}\cfrac{1}{2L_i}+\rho_i\right)^2}\norm{z_0-z^*}^2,\quad\forall k\geqslant 1.
	\]
\end{corollary}

\begin{proof}
	Consider the Lyapunov function $\E(k)$ defined in Theorem \ref{thm:linesearchseg+}. Let $z^*$ in $\E(k)$ be the point defined in Corollary \ref{coro:linesearchseg+}. By the upper bound of $\E(k+1)-\E(k)$ in Theorem \ref{thm:linesearchseg+} , the assumption $D<2(r-1)$ and $\inner{F(z_{k+1})+\tilde{G}_{k+1},z_{k+1}-z^*}\geqslant\rho_k\norm{F(z_{k+1})}^2$, we have\[
	\E(k+1)-\E(k)\leqslant 0.
	\]
	Next, we deduce the lower bound of $\E(k)$. Since $\inner{F(z_{k+1})+\tilde{G}_{k+1},z_{k+1}-z^*}\geqslant\rho_k\norm{F(z_{k+1})+\tilde{G}_{k+1}}^2$, we have\begin{align*}
		\E(k) \geqslant&  \left(D-\frac{D^2}{2(r-1)}\right)\left(\sum_{i=0}^{k-1}\frac{1}{2L_i}+\rho_i\right)^2\norm{F(z_k)+\tilde{G}_k}^2\\
		&\ +\frac{1}{2}\norm{\frac{D}{\sqrt{r-1}}\left(\sum_{i=0}^{k-1}\frac{1}{2L_i}+\rho_i\right)F(z_k)-\sqrt{r^3-r^2}(u_k-z^*)}^2.
	\end{align*}
	Thus
	\begin{align*}
		&\ \left(D-\frac{D^2}{2(r-1)}\right)\left(\sum_{i=0}^{k-1}\frac{1}{2L_i}+\rho_i\right)^2\dist(0,T(z_k))^2\\
		\leqslant&\ \left(D-\frac{D^2}{2(r-1)}\right)\left(\sum_{i=0}^{k-1}\frac{1}{2L_i}+\rho_i\right)^2\norm{F(z_k)+\tilde{G}_k}^2\\
		\leqslant&\ \E(k)\leqslant\E(0)=\frac{r^3-r^2}{2}\norm{z_0-z^*}^2.
	\end{align*}
	By dividing $\left(D-\dfrac{D^2}{2(r-1)}\right)\left(\sum\limits_{i=0}^{k-1}\dfrac{1}{2L_i}+\rho_i\right)^2$ on both sides of above inequality, we obtain the desired result.
\end{proof}

However, the assumption $\inner{F(z_{k+1})+\tilde{G}_{k+1}, z_{k+1}-z^*}\geqslant\rho_k\norm{F(z_{k+1})}^2$ can not be determined in practice. To overcome such problem, we introduce an upper bound constraint on $\rho_k$.
\begin{corollary}
	\label{coro:linesearchseg+constant}
	Let $\{z_{\frac{k}{2}}\}$, $\{\tilde{z}_{k+1}\}$ and $\{u_k\}$ be the sequences generated by Algorithm \ref{al:linesearchseg+}. If the parameters $\{L_k\}$, $\{\rho_k\}$ and $D$ satisfy the assumptions in Theorem \ref{thm:linesearchseg+} and there exists a $z^*$ such that\[
	\begin{gathered}
		\inner{F(z_k)+\tilde{G}_k, z-z^*}\geqslant\tilde{\rho}_k\norm{F(z_k)+\tilde{G}_k}^2,\quad\forall z, \\
		\rho_k-\tilde{\rho}_k\leqslant\frac{E_k}{r}\left(1-\frac{D}{2(r-1)}\right)\left(\sum_{i=0}^{k}\frac{1}{2L_i}+\rho_i\right),\quad  E_k\in[0,1],
	\end{gathered}	
	\]
	then we have
	\begin{align*}
		&\ \sum_{k=0}^\infty \frac{D}{2}\left(\frac{1}{2L_k}+\rho_k\right)\left(r-1-\frac{D}{2}\right)\left[(r-1)\left(\frac{1}{2L_k}+\rho_k\right)+2(1-E_k)\Sigma_{k+1}\right]\dist(0,T(z_{k+1}))^2\\
		\leqslant&\ \frac{r^3-r^2}{2}\cdot\norm{z_0-z^*}^2,\\
		&(1-E_k)\left(1-\frac{D}{2r-2}\right)\Sigma_k^2\dist(0,T(z_k))^2\leqslant\frac{r^3-r^2}{2D}\cdot\norm{z_0-z^*}^2,
	\end{align*}
	where $\Sigma_k=\sum\limits_{i=0}^{k-1}\dfrac{1}{2L_i}+\rho_i$.
\end{corollary}

\begin{proof}
	For the sake of simplicity, let $\tilde{T}(z_k)=F(z_k)+\tilde{G}_k$. Let $z^*$ in Theorem \ref{thm:linesearchseg+} be the point such that 
	$\inner{\tilde{T}(z_k),z_k-z^*}\geqslant\tilde{\rho}_k\norm{\tilde{T}(z_k)}^2$, $\forall k$. Then we have\begin{align*}
		\E(k+1)-\E(k)\leqslant&\ -\frac{D}{2}\left(\frac{1}{2L_k}+\rho_k\right)\left(r-1-\frac{D}{2}\right)\left[(r+1)\left(\frac{1}{2L_k}+\rho_k\right)+2\Sigma_{k+1}\right]\norm{\tilde{T}(z_{k+1})}^2\\
		&\ +D(r^2-r)\left(\frac{1}{2L_k}+\rho_k\right)(\rho_k-\tilde{\rho}_k)\norm{\tilde{T}(z_{k+1})}^2.
	\end{align*}
	By the assumption that $\rho_k-\tilde{\rho}_k\leqslant\dfrac{E}{r}\left(1-\dfrac{D}{2(r-1)}\right)\Sigma_{k+1}$, we have\begin{align*}
		&\ \E(k+1)-\E(k)\\
		\leqslant&\ -\frac{D}{2}\left(\frac{1}{2L_k}+\rho_k\right)\left(r-1-\frac{D}{2}\right)\left[(r-1)\left(\frac{1}{2L_k}+\rho_k\right)+2(1-E_k)\Sigma_{k+1}\right]\norm{\tilde{T}(z_{k+1})}^2.
		zs\end{align*}
	In addition, the lower bound of $\E(k)$ can be estimated as follows:\[
	\E(k)\geqslant (1-E_k)D\left(1-\frac{D}{2r-2}\right)\Sigma_k^2\norm{\tilde{T}(z_k)}^2+\frac{1}{2}\norm{\frac{D}{\sqrt{r-1}}\Sigma_k\tilde{T}(z_k)-\sqrt{r^3-r^2}(u_k-z^*)}^2\geqslant 0.
	\]
	Thus, we have\begin{align*}
		&\ (1-E_k)\left(1-\frac{D}{2r-2}\right)\Sigma_k^2\dist(0,T(z_k))^2 \\
		\leqslant &\ (1-E_k)\left(1-\frac{D}{2r-2}\right)\Sigma_k^2\norm{\tilde{T}(z_k)}^2\\
		\leqslant&\ \E(k)\leqslant\E(0)=\frac{r^3-r^2}{2}\norm{z_0-z^*}^2.
	\end{align*}
	Also, by summing $\E(k+1)-\E(k)$ from 0 to $\infty$, we can obtain the ergodic convergence rate of Algorithm \ref{al:linesearchseg+}.
\end{proof}

The line search framework of the SPEG+ method is similar to Algorithm \ref{al:linesearchseg+}, which is described as follows along with its convergence results.
\begin{corollary}[SPEG+ with line search]
	\label{coro:linesearchspeg+}
	Let $\{z_{\frac{k}{2}}\}$, $\{\tilde{z}_{k+1}\}$ and $\{u_k\}$ be the sequences generated by the following recursive rule:\begin{align*}
		\tilde{z}_{k+1} &= (1-\alpha_k)z_k+\alpha_ku_k,\\
		z_{k+\frac{1}{2}} &= P_C\left(\tilde{z}_{k+1}-(1-\alpha_k)\frac{1}{L_k}F(z_k)\right),\\
		z_{k+1} &= P_C\left(\tilde{z}_{k+1}-\dfrac{1}{L_k}F(z_{k+\frac{1}{2}})\right),\\
		\tilde{C}_{k+1}  &= L_k(\tilde{z}_{k+1}-z_{k+1})-F(z_{k+\frac{1}{2}}),\\
		u_{k+1} &= u_k-\dfrac{D}{2rL_k}[F(z_{k+1})+\tilde{C}_{k+1}].
	\end{align*}
	If $\alpha_k=\dfrac{rL_k^{-1}}{\sum_{i=0}^{k-1}L_i^{-1}+rL_k}$, $r>1$, $0<D<2(r-1)$ and $L_k\geqslant\dfrac{\norm{F(z_{k+\frac{1}{2}})-F(z_{k+1})}}{\norm{z_{k+\frac{1}{2}}-z_{k+1}}}$, then we have\begin{align*}
		\dist(0,T(z_k))^2 & \leqslant \cfrac{r^2(r-1)^2}{[2(r-1)D-D^2]\left(\sum\limits_{i=0}^{k-1}\cfrac{1}{2L_i}\right)^2}\cdot\dist(z_0,\zero(T))^2,\quad\forall k\geqslant 1,\\
		\inner{F(z_k)+\tilde{C}_k, z_k-z^*} & \leqslant\frac{r^2-r}{Dr\sum_{i=0}^{k-1}{L_i}^{-1}}\norm{z_0-z^*}^2,\quad\forall z^*\in\zero(T).
	\end{align*}
\end{corollary}

The proof for Corollary \ref{coro:linesearchspeg+} is similar to the proof for Theorem \ref{thm:linesearchseg+} and Corollary \ref{coro:linesearchseg+}. Thus, we omit the proof of Corollary \ref{coro:linesearchspeg+}.

\noindent\textbf{Remark. } 1) Although previous arguments suggest that when $D=r-1$, the corresponding instance of Algorithm \ref{al:linesearchseg+} may obtain the fastest numerical performance, this result is not generally observed in practical numerical experiment. The reason is that the inequalities
\begin{align*}
	\rho_k&\leqslant\frac{\inner{F(z_{k+1})+\tilde{G}_{k+1}-F(z_k)+\tilde{G}_k, z_{k+1}-z_k}}{\norm{F(z_{k+1})+\tilde{G}_{k+1}-F(z_k)-\tilde{G}_k}^2},\\
	L_k&\geqslant\frac{\norm{F(z_{k+\frac{1}{2}})-F(z_{k+1})}}{\norm{z_{k+\frac{1}{2}}-z_{k+1}}}
\end{align*}
are generally strict. Thus we encourage that $D>r-1$ when using Algorithm \ref{al:linesearchseg+}.

2) In Theorem \ref{thm:linesearchseg+}, Algorithm \ref{al:linesearchseg+} imposes a condition that the parameter $\rho_k$ is non-decreasing. Nonetheless, a problem may occur in practice, that is $\rho_k$ may grow too large so that there does not exist a feasible point $z_{k+1}$ such that the line search principle can be fulfilled. To circumvent this problem, a simple solution is restarting the algorithm. Specifically, run Algorithm \ref{al:linesearchseg+} again with initial point $z_k$.
\section{Numerical Experiments}
\label{sec:6}
\subsection{Matrix Game Problem}

Here, we consider the matrix game problem illustrated in Example \ref{ex:matrixgame}. In our experiment, we evaluate the numerical performance of the SPEG+ method with line search. Specifically, the matrix $A \in \mathbb{R}^{1000 \times 1000}$ is randomly generated and the parameter $D$ in Algorithm \ref{al:linesearchseg+} is set to $1.6$. For comparison, both the Projected EG method with line search and the Projected FEG method with line search are included in this experiment. Each algorithm runs for a duration of 40 seconds. Subsequently, we depict the squared norm of $\tilde{T}(z_k) = F(z_k) + \tilde{G}_k$ and the duality gap over time in Figure \ref{fig:matrixgame}. Here, the duality gap is defined as $\max_{y \in \Delta_n} x_k^\top A y - \min_{x \in \Delta_m} x^\top A y_k$.

\begin{figure}[ht]
	\centering
	\subfigure{\includegraphics[width=.4\textwidth]{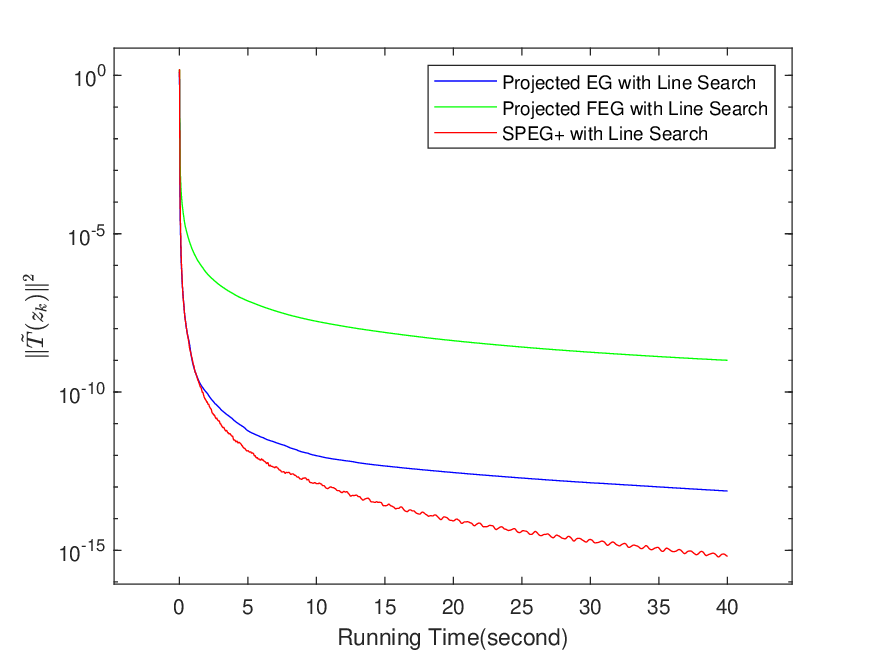}}
	\subfigure{\includegraphics[width=.4\textwidth]{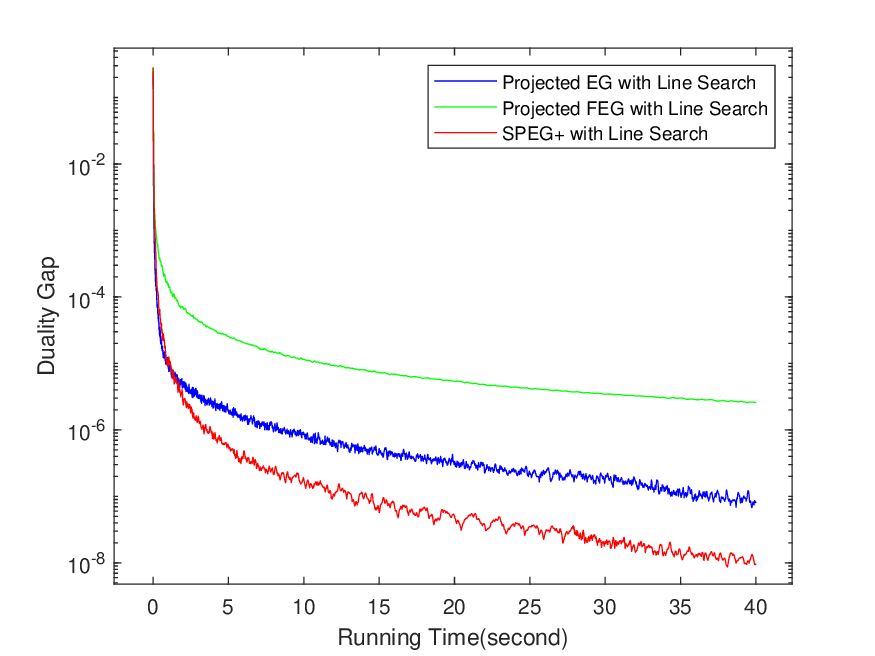}}
	\caption{Evolution of duality gap and $\norm{\tilde{T}(z_k)}^2$ versus running time.}
	\label{fig:matrixgame}
\end{figure}

As we can see from Figure \ref{fig:matrixgame}, in comparison to the projected EG method with line search and the projected FEG method with line search, the SPEG+ method with line search converges fastest. 

Since the normal cone of a closed convex set is a maximally monotone operator, we are interested in examining the numerical performance of the SPEG+ method and the SFBS method. The test problem remains the matrix game. Following a 40-second runtime for each algorithm,  we plot $\norm{\tilde{T}(z_k)}^2$ and duality gap with respect to running time in Figure \ref{fig:SFBSvsSPEG+}, too.
\begin{figure}[ht]
	\centering
	\subfigure{\includegraphics[width=.4\textwidth]{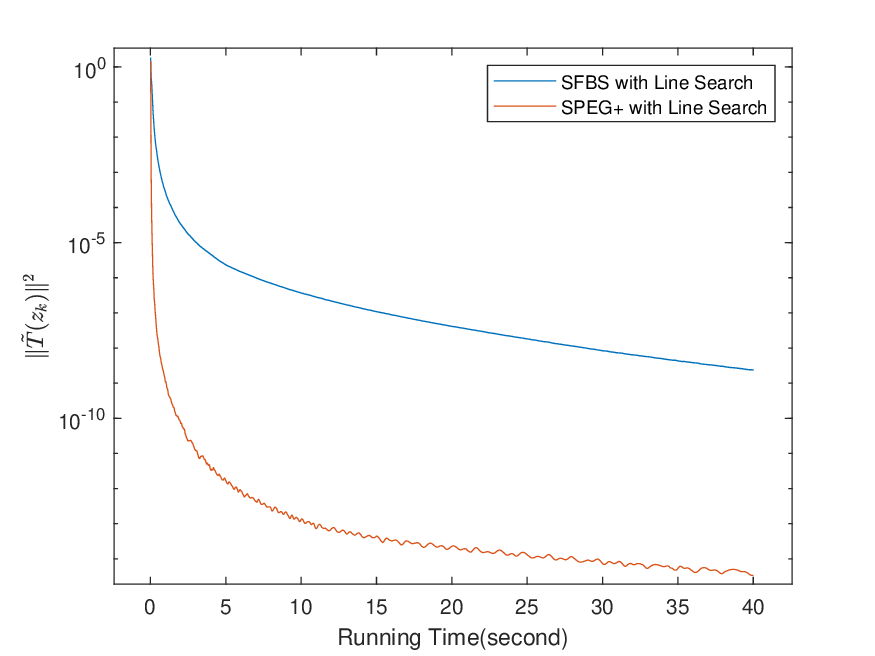}}
	\subfigure{\includegraphics[width=.4\textwidth]{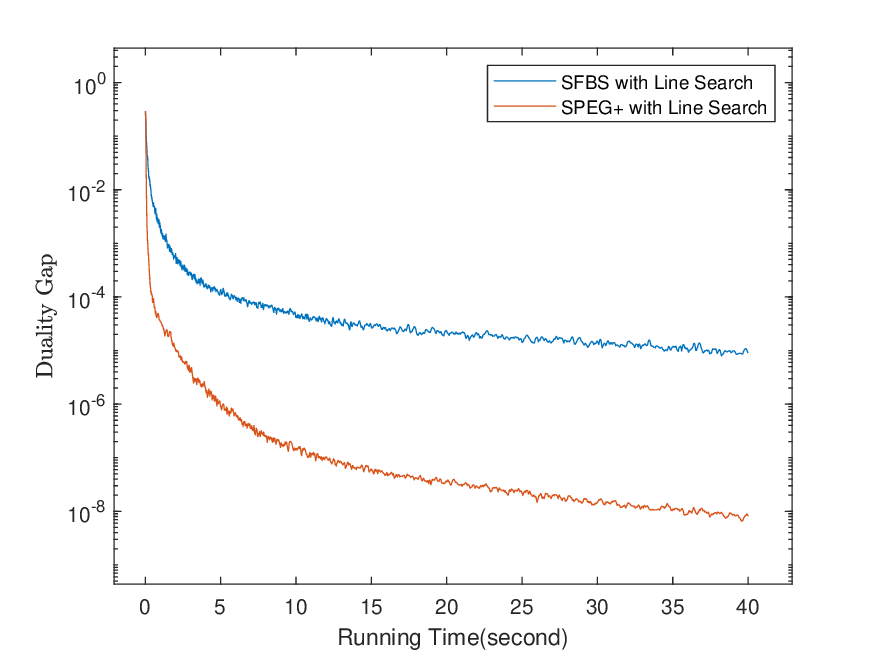}}
	\caption{Comparison of SFBS and SPEG+.}
	\label{fig:SFBSvsSPEG+}
\end{figure}

As Figure \ref{fig:SFBSvsSPEG+} illustrates, the SPEG+ method converges faster than the SFBS method in practice. This suggests that the SPEG+ method may better exploit the structure of \eqref{eq:mainconvexset}. Consequently, we recommend utilizing the SPEG+ method for solving \eqref{eq:mainconvexset} when projections onto $C$ can be easily computed.
\subsection{Lasso Problem}
Here, we consider the LASSO problem\begin{equation}
	\min_x F(x):=\frac{1}{2}\norm{Ax-b}^2+\mu\norm{x}_1.
	\label{eq:lasso}
\end{equation}
\eqref{eq:lasso} can be transformed into the following constrained optimization problem\[
\begin{gathered}
	\min_{x,y} \frac{1}{2}\norm{Ax-b}^2+\mu\norm{y}_1\\
	s.\ t.\quad x=y.
\end{gathered}
\]
The ADMM method is a powerful method to solve LASSO problem and the convergence rate of ADMM was been proven to be $O(1/k^2)$ when applied to LASSO problem in the research work \cite{tian19}. Since the ADMM method is an instance of the proximal point algorithm, we can use Algorithm \ref{al:linesearchseg+} with $F=0$ and $L_k\equiv 1$ to obtain a new ADMM type method with line search. In our experiment, we test the numerical performance of the new ADMM method with line search. Here, the matrix $A\in\mathbb{R}^{1000\times 2000}$ and the vector $b\in\mathbb{R}^{1000}$ are randomly generated, and the parameter $D$ in Algorithm \ref{al:linesearchseg+} is set to be $1.6$. The ADMM method is comparable in our experiment. After running the algorithms, we graph the residual $\norm{x_k-y_k}$ and $F(y_k)-F^*$ respect to running time in the following figure, where $F^*$ represents the optimal value of the LASSO problem \eqref{eq:lasso} as determined by an alternative algorithm.

\begin{figure}[ht]
	\centering
	\subfigure{\includegraphics[width=.4\textwidth]{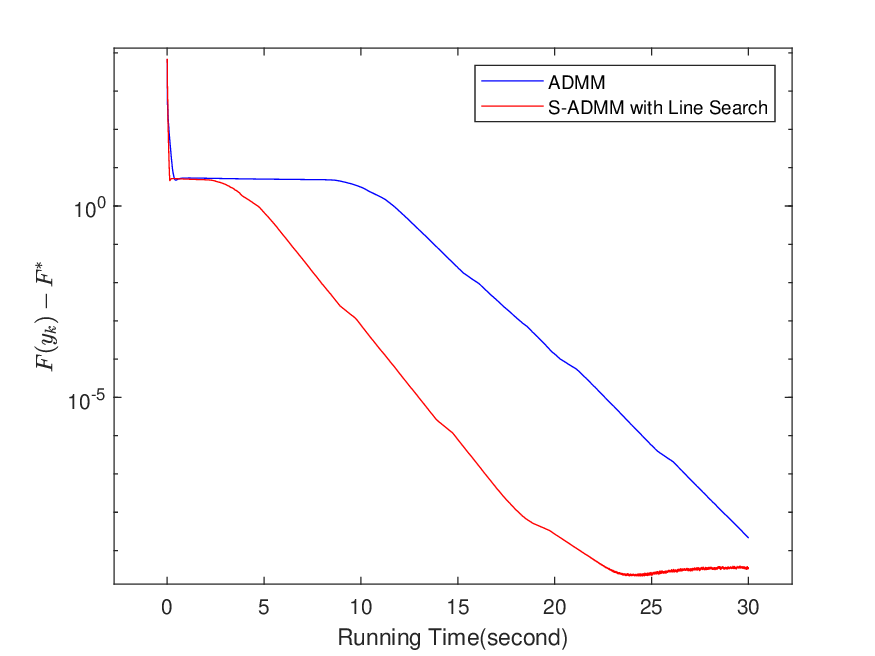}}
	\subfigure{\includegraphics[width=.4\textwidth]{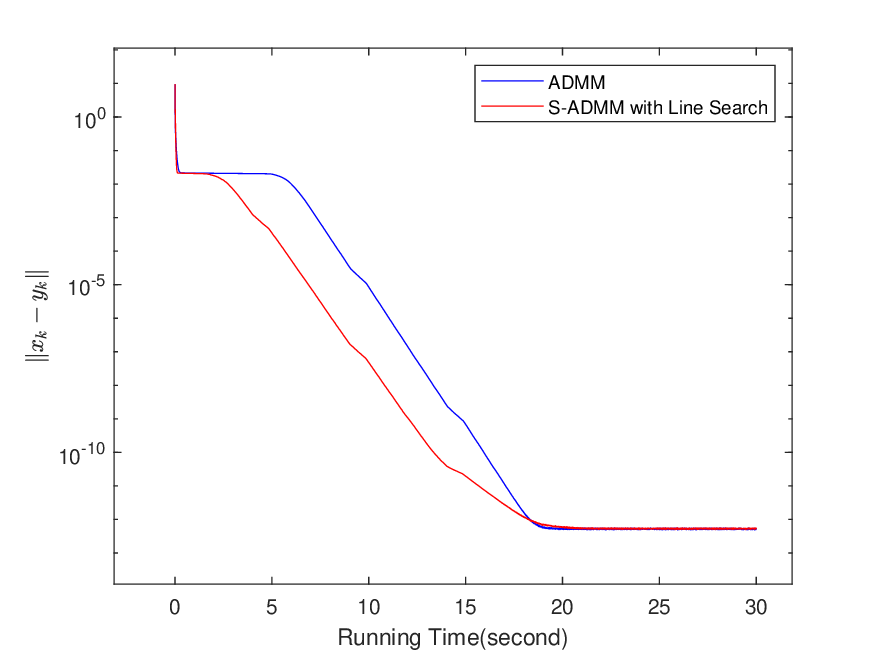}}
	\caption{Evolution of $\norm{x_k-y_k}$ and $F(y_k)-F^*$ versus running time.}
	\label{fig:admm}
\end{figure}

From Figure \ref{fig:admm}, it is evident that despite the theoretically established $O(1/k^2)$ convergence rate of ADMM, the  new ADMM derived by Algorithm \ref{al:linesearchseg+} exhibits even faster convergence in our experiment.
\section{Conclusion and Further Discussion}

\label{sec:7}
In our paper, we introduce the symplectic extra-gradient type methods based on the symplectic proximal point algorithm. By using the Lyapunov function approach, we show these methods have a convergence rate of $O(1/k^2)$. We also prove they can achieve a faster $o(1/k^2)$ convergence rate and show weak convergence property under stronger assumptions. To improve computational speed, we incorporate a line search technique into our symplectic extra-gradient method.

Several theoretical questions around SEG type method remain. With the heightened need for efficient Generative Adversarial Networks training, studying a stochastic version of our method is intriguing. Furthermore, inspired by Tseng's method for more general monotone inclusion problems discussed in \cite{bauschke17}, developing a more general SEG type method is of interest. Finally, we are intrigued to explore whether symplectic acceleration can be useful to derive a broader range of acceleration forms of algorithms that designed to solve either minimax problems or inclusion problems. 

Practically, creating efficient algorithms to find good values for $L_k$ and $\rho_k$ in Algorithm \ref{al:linesearchseg+}, or just $L_k$ as in Corollary \ref{coro:linesearchspeg+}, is a concern. Our LASSO problem tests indicate that the current line search for $\rho_k$ does not speed up calculations significantly. One suggestion from our analysis in Corollary \ref{coro:linesearchseg+constant} is that $\rho_k$ shouldn't greatly exceed the comonotone index. Therefore, improving when to stop adjusting $\rho_k$	could be beneficial.

\appendix

\section{Convergence of Symplectic Extra-gradient Method}
\label{app:segrate}
To proof convergence of the symplectic extra-gradient method, we need to study the following Lyapunov function:
\begin{equation}
	\label{eq:lyapunovseg}
	\E(k)=\frac{A_k}{2}\norm{F(z_k)}^2+B_k\inner{F(z_k),z_k-u_k}.
\end{equation}
The first thing we should do is to find out the sufficient condition to ensure that $\E$ is a Lyapunov function, i. e. $\{\E(k)\}$ is non-increasing. The conditions \eqref{eq:thm1assumption1}-\eqref{eq:thm1assumption3} are based on the conditions in Lemma 2 in \cite{yoon21}.

\begin{theorem}
	\label{thm:lyapunov}
	Let $\{z_\frac{k}{2}\}$, $\{\tilde{z}_k\}$ and $\{u_k\}$ be the sequences generated by \eqref{eq:seg}. If\begin{align}
		A_k&=\frac{B_k\beta_k}{\alpha_k},\label{eq:thm1assumption1}\\
		B_{k+1}&=\frac{B_k}{1-\alpha_k},\label{eq:thm1assumption2}\\
		\beta_{k+1}+2C_k\alpha_{k+1}&=\frac{\beta_k\alpha_{k+1}(1-L^2\beta_k^2-\alpha_k^2)}{\alpha_k(1-\alpha_k)(1-L^2\beta_k^2)},\label{eq:thm1assumption3}\\
		\beta_k&\in(0,\frac{1}{L}),\quad\forall k\geqslant 0.\label{eq:thm1assumption4}\,
	\end{align}
	then the sequence $\{\E(k)\}$ is non-increasing. 
\end{theorem}
\begin{proof}
	Step 1: Dividing the difference $\E(k+1)-\E(k)$ into 2 parts.\begin{align*}
		\E(k+1)-\E(k)=&\  \underbrace{\frac{A_{k+1}}{2}\norm{F(z_{k+1})^2}-\frac{A_k}{2}\norm{F(z_k)}^2}_{\text{I}}\\
		&\ +\underbrace{B_{k+1}\inner{F(z_{k+1}),z_{k+1}-u_{k+1}}-B_k\inner{F(z_k),z_k-u_k}}_{\text{II}}.
	\end{align*}
	
	Step 2: Estimate the upper bound of II.  II can be divided into the following three pieces.\begin{align*}
		&\ B_{k+1}\inner{F(z_{k+1}),z_{k+1}-u_{k+1}}-B_k\inner{F(z_k),z_k-u_k}\\
		=&\ (B_{k+1}-B_k)\inner{F(z_{k+1}),z_{k+1}-u_{k+1}}+B_k\inner{F(z_{k+1}),z_{k+1}-z_k-u_{k+1}+u_k}\\
		&+B_k\inner{F(z_{k+1})-F(z_k),z_k-u_k}.
	\end{align*}
	Since\[
	z_{k+1}=\alpha_ku_k+(1-\alpha_k)z_k-\beta_kF(x_{k+\frac{1}{2}}),
	\] 
	we have
	\begin{equation}
		B_k\inner{F(z_{k+1})-F(z_k),z_k-u_k}=B_k\inner{F(z_{k+1})-F(z_k),-\frac{1}{\alpha_k}(z_{k+1}-z_k)-\frac{\beta_k}{\alpha_k}F(z_{k+\frac{1}{2}})}.
		\label{eq:thm1_a}
	\end{equation}
	Also, we have
	\begin{equation}
		\begin{split}
			&\ B_k\inner{F(z_{k+1}),z_{k+1}-z_k-u_{k+1}+u_k}\\
			=&\ B_k\inner{F(z_{k+1}),\frac{\alpha_k}{1-\alpha_k}(u_k-z_{k+1})-\frac{\beta_k}{1-\alpha_k}F(z_{k+\frac{1}{2}})}+B_kC_k\norm{F(z_{k+1})}^2.
		\end{split}
		\label{eq:thm1_b}
	\end{equation}
	In addition, we have
	\begin{equation}
		(B_{k+1}-B_k)\inner{F(z_{k+1}),z_{k+1}-u_{k+1}}=(B_{k+1}-B_k)C_k\norm{F(z_{k+1})}^2+(B_{k+1}-B_k)\inner{F(z_{k+1}),z_{k+1}-u_k}.
		\label{eq:thm1_c}
	\end{equation}
	By summing \eqref{eq:thm1_a}, \eqref{eq:thm1_b} and \eqref{eq:thm1_c} and using \eqref{eq:thm1assumption2}, we obtain\begin{align*}
		\text{II}=&\ B_{k+1}C_k\norm{F(x_{k+1})}^2-\frac{B_k}{a_k}\inner{F(z_{k+1})-F(z_k),z_{k+1}-z_k}\\
		&\ -\left(\frac{B_k\beta_k}{\alpha_k}+\frac{B_k\beta_k}{1-\alpha_k}\right)\inner{F(z_{k+1}),F(z_{k+\frac{1}{2}})}+\frac{B_k\beta_k}{\alpha_k}\inner{F(z_k),F(z_{k+\frac{1}{2}})}.
	\end{align*}
	Since $F$ is $L-$Lipschitz continuous, we have
	\begin{align*}
		\norm{F(z_{k+1})-F(z_{k+\frac{1}{2}})}^2&\leqslant L^2\beta_k^2\norm{F(z_{k+\frac{1}{2}})-F(z_k)}^2.
	\end{align*}
	Because of the equality\[
	-\inner{F(z_{k+1}),F(z_{k+\frac{1}{2}})}=\frac{1}{2}\left[\norm{F(z_{k+1})-F(z_{k+\frac{1}{2}})}^2-\norm{F(z_{k+1})}^2-\norm{F(z_{k+\frac{1}{2}})}^2\right],
	\]
	we have\begin{align*}
		\text{II}\leqslant& \left(B_{k+1}C_k-\frac{B_k\beta_k}{2\alpha_kL^2\beta_k^2}\right)\norm{F(z_{k+1})}^2+\frac{B_k\beta_k}{2\alpha_k}\norm{F(z_k)}^2+\frac{B_k\beta_k}{2\alpha_k}(1-\frac{1}{L^2\beta_k^2})\norm{F(z_{k+\frac{1}{2}})}^2\\
		&-\left(\frac{B_k\beta_k}{\alpha_k}(1-\frac{1}{L^2\beta_k^2})+\frac{B_k\beta_k}{1-\alpha_k}\right)\inner{F(z_{k+\frac{1}{2}}),F(z_{k+1})}.
	\end{align*}

	Step 3: Estimate the upper bound of $\E(k+1)-\E(k)$. By previous estimation and the equality \eqref{eq:thm1assumption1}, we have\begin{align*}
		\E(k+1)-\E(k)\leqslant &\  \left(\frac{A_{k+1}}{2}+B_{k+1}C_k-\frac{B_k\beta_k}{2\alpha_kL^2\beta_k^2}\right)\norm{F(z_{k+1})}^2+\frac{B_k\beta_k}{2\alpha_k}(1-\frac{1}{L^2\beta_k^2})\norm{F(z_{k+\frac{1}{2}})}^2\\
		&\ -\left(\frac{B_k\beta_k}{\alpha_k}(1-\frac{1}{L^2\beta_k^2})+\frac{B_k\beta_k}{1-\alpha_k}\right)\inner{F(z_{k+\frac{1}{2}}),F(z_{k+1})}.
	\end{align*}
	By \eqref{eq:thm1assumption1}, \eqref{eq:thm1assumption2} and \eqref{eq:thm1assumption3}, we have\begin{align*}
		\frac{A_{k+1}}{2}+B_{k+1}C_k & = \frac{B_{k+1}\beta_{k+1}}{2\alpha_{k+1}}+B_{k+1}C_k \\
		& = \frac{B_k\beta_{k+1}}{2\alpha_{k+1}(1-\alpha_k)}+\frac{B_kC_k}{1-\alpha_k}\\
		& =\frac{B_k(\beta_{k+1}+2C_k\alpha_{k+1})}{2\alpha_{k+1}(1-\alpha_k)}\\
		& =\frac{B_k\beta_k(1-L^2\beta_k^2-\alpha_k^2)}{2\alpha_k(1-\alpha_k)^2(1-L^2\beta_k^2)}\\
		&= \frac{A_k(1-L^2\beta_k^2-\alpha_k^2)}{2(1-\alpha_k)^2(1-L^2\beta_k^2)}.
	\end{align*}
	Also by \eqref{eq:thm1assumption1}, we have\begin{align*}
		\frac{B_k\beta_k}{2\alpha_k}(1-\frac{1}{L^2\beta_k^2})&=\frac{A_k(L^2\beta_k^2-1)}{2L^2\beta_k^2},\\
		-\frac{B_k\beta_k}{\alpha_k}(1-\frac{1}{L^2\beta_k^2})+\frac{B_k\beta_k}{1-\alpha_k}&=\frac{A_k(1-L^2\beta_k^2-\alpha_k)}{2(1-\alpha_k)L^2\beta_k^2}.
	\end{align*}
	In conclusion, we have\begin{align*}
		\E(k+1)-\E(k) \leqslant&\ -\frac{A_k(1-L^2\beta_k^2-\alpha_k)^2}{2(1-\alpha_k)^2(1-L^2\beta_k^2)L^2\beta_k^2}\norm{F(z_{k+1})}^2 - \frac{A_k(1-L^2\beta_k^2)}{2L^2\beta_k^2}\norm{F(z_{k+\frac{1}{2}})}^2\\
		&\  +\frac{A_k(1-L^2\beta_k^2-\alpha_k)}{2(1-\alpha_k)L^2\beta_k^2}\inner{F(z_{k+\frac{1}{2}}),F(z_{k+1})}\\
		=&\ -\frac{A_k}{2L^2\beta_k^2}\norm{\frac{(1-L^2\beta_k^2-\alpha_k)}{(1-\alpha_k)\sqrt{1-L^2\beta_k^2}}F(z_{k+1})-\sqrt{1-L^2\beta_k^2}F(z_{k+\frac{1}{2}})}^2\\
		\leqslant&\ 0.
	\end{align*}
\end{proof}

However, unlike Theorem 2 in \cite{yoon21}, we can not directly estimate the upper bounds of $\norm{F(z_k)}^2,$ and $\inner{F(z_k),z_k-z^*}$. The main reason is that the lower bound of $E(k)$ is\[
\E(k)\geqslant\frac{A_k}{4}\norm{F(z_k)}^2+B_k\inner{F(z_k),z_k-z^*}-\frac{B_k^2}{A_k^2}\norm{u_k-z^*}^2,
\]
which leads to that both $\norm{F(z_k)}^2,$ and $\inner{F(z_k),z_k-z^*}$ are bounded by both $\norm{u_k-z^*}^2$ and $\norm{z_0-z^*}^2$. Because determining all classes of SEG that satisfies \eqref{eq:thm1assumption4} and admits $O(1/k^2)$ convergence rate is complicated, we first focus on the following instance of SEG, described in Algorithm \ref{al:segvary}.

\begin{algorithm}[ht]
	\label{al:segvary}
	\caption{SEG with Varying Step-size}
	\textbf{Input: }Operator $F$, $L\in(0,+\infty)$, \;
	\textbf{Initialization: } $z_0$, $u_0=z_0$\;
	\For{$k=0, 1, \cdots$}{
		$\tilde{z}_{k+1}=\dfrac{1}{k+r}z_k+\dfrac{k+r-1}{k+r}u_k$\;
		$z_{k+\frac{1}{2}}=\tilde{z}_{k+1}-\beta_k F(z_k)$\;
		$z_{k+1}=\tilde{z}_{k+1}-\beta_k F(z_{k+\frac{1}{2}})$\;
		$u_{k+1}=u_k-D\cdot\dfrac{L^2\beta_k^3}{2(k+r-1)(1-L^2\beta_k^2)}F(z_{k+1})$\;
		$\beta_{k+1}=\beta_k-(1+D)\cdot\dfrac{L^2\beta_k^3}{[(k+r)^2-1](1-L^2\beta_k^2)}$\;
	}
\end{algorithm}

Algorithm \ref{al:segvary} can be seen as SEG with $\alpha_k=\dfrac{1}{k+r}, C_k=D\cdot\dfrac{L^2\beta_k^3}{2(k+r-1)(1-L^2\beta_k^2)}$. The first thing we need to do is propose the sufficient condition for SEG such that \eqref{eq:thm1assumption4} hold.

\begin{lemma}
	\label{lem:1}
	Let $\{z_\frac{k}{2}\}$, $\{\tilde{z}_k\}$ and $\{u_k\}$ be the sequences generated by Algorithm \ref{al:segvary}. If 
	\begin{align}
		\beta_0&\in (0,\dfrac{1}{L}),\label{eq:lem1_a}\\
		-1<D&<\frac{r(r-1)(1-L^2\beta_0^2)}{(2r-1)L^2\beta_0^2}-1\label{eq:lem1_b}
	\end{align}
	hold, $\{\beta_k\}$ is a non-increasing sequence and $\beta_\infty=\lim\limits_{k\to\infty}\beta_k>0$.
\end{lemma}
\begin{proof}
	Without loss of generalization, we presume $L=1$ because the general case can be obtained by replacing $\beta_k$ by $L\beta_k$.  By the assumption, the difference $\beta_{k+1}-\beta_k$ is given as follows:\[
	\beta_{k+1}-\beta_k=-(1+D)\cdot\frac{\beta_k^3}{[(k+r)^2-1](1-\beta_k^2)}.
	\]
	
	Next, we verify that $\beta_k>0$ by using mathematical induction. We assume that $\beta_i>0$ is hold for all $i=0, \cdots, k$, then we have $\beta_{i+1}-\beta_i\leqslant 0, \forall i=0, \cdots, k$. Since the function $g(x)=\dfrac{x^3}{1-x^2}$ is non-increasing on $(0,1)$, we have\[
	\frac{\beta_i^3}{1-\beta_i^2}\leqslant\frac{\beta_0^3}{1-\beta_0^2}.
	\] 
	Thus\begin{align*}
		\beta_{k+1}-\beta_0 & =\sum_{i=0}^{k}\beta_{i+1}-\beta_i \\
		& \geqslant-(1+D)\cdot\sum_{i=0}^{k}\frac{\beta_0^3}{(i+r-1)(i+r-1)(1-\beta_0^2)}\\
		&\geqslant-\frac{(1+D)\beta_0^3}{(1-\beta_0^2)}\sum_{i=0}^{\infty}\frac{1}{(i+r-1)(i+r+1)}\\
		&=-\frac{(2r-1)(1+D)\beta_0^3}{r(r-1)(1-\beta_0^2)}.
	\end{align*}
	In conclusion, we have\[
	\beta_{k+1}\geqslant\left(1-\frac{(2r-1)(1+D)\beta_0^2}{r(r-1)(1-\beta_0^2)}\right)\beta_0.
	\]
\end{proof}
The following lemma is useful to show that $\norm{u_k-z^*}^2$ can be uniformly bounded by $\norm{z_0-z^*}^2$.

\begin{lemma}
	\label{lem:2}
	Let $\{a_k\}$ be a non-negative sequence. If there exists two positive constants $E_1, E_2$ such that \[
	\left(1-\frac{E_1}{(k+r)^2-1}\right)a_{k+1}\leqslant\frac{E_2}{(k+r-1)(k+r)}a_0+a_k,\quad E_1<r^2-1,
	\]
	the sequence $\{a_k\}$ can be uniformly bounded by $a_0$, i. e. there exists a positive constant $E_3$ such that\[
	a_k\leqslant E_3a_0,\quad\forall k\geqslant 1.
	\]
\end{lemma}

\begin{proof}
	Deriving $1-\dfrac{E_1}{(k+r)^2-1}$ on both sides of the inequality in Lemma \ref{lem:2}, we have\[
	a_{k+1}\leqslant\frac{(k+r+1)E_2}{[(k+r-1)^2-E_1](k+r)}a_0+\frac{(k+r)^2-1}{(k+r)^2-1-E_1}a_k.
	\]
	Recursively using upper bound of $a_k, a_{k-1}, \cdots, a_1$, we have\[
	a_{k+1}\leqslant\sum_{i=0}^{k}\prod_{j=i+1}^{k}\frac{(j+r)^2-1}{(j+r)^2-1-E_1}\frac{(i+r+1)E_2}{[(i+r)^2-1-E_1](i+r)}a_0,
	\]
	where $\prod\limits_{j=k+1}^{k}\dfrac{(j+r)^2-1}{(j+r)^2-1-E_1}=1$. Since\[
	\prod_{j=i+1}^{k}\frac{(j+r)^2-1}{(j+r)^2-1-E_1}\leqslant\prod_{j=0}^{\infty}\frac{(j+r)^2-1}{(j+r)^2-1-E_1}<\infty,
	\]
	then we have\[
	a_{k+1}\leqslant E_4\frac{(i+r+1)E_2}{[(i+r)^2-1-E_1](i+r)}a_0,
	\]
	where\[
	E_4=\prod_{j=0}^{\infty}\frac{(j+r)^2-1}{(j+r)^2-1-E_1}.
	\]
	Also, because\[
	\sum_{i=0}^{k}\frac{(i+r+1)E_2}{[(i+r)^2-1-E_1](i+r)}\leqslant\sum_{i=0}^{\infty}\frac{(i+r+1)E_2}{[(i+r)^2-1-E_1](i+r)}<\infty,
	\]
	we have\[
	a_{k+1}\leqslant E_3 a_0
	\]
	where \[
	E_3=E_4E_2\sum_{i=0}^{\infty}\frac{i+r+1}{[(i+r)^2-1-E_1](i+r)}.
	\]
\end{proof}
With Lemma \ref{lem:1} and Lemma \ref{lem:2}, we can show that Algorithm \ref{al:segvary} exhibits $O(1/k^2)$ convergence rates.
\begin{theorem}
Let $\{z_\frac{k}{2}\}$, $\{\tilde{z}_k\}$ and $\{u_k\}$ be the sequences generated by Algorithm \ref{al:segvary}. If \eqref{eq:lem1_a} and \eqref{eq:lem1_b} hold, we have\begin{align*}
	\norm{F(z_k)}^2 & \leqslant O\left(\frac{\dist(z_0,F)^2}{(k+r-1)(k+r)}\right), \\
	\inner{F(z_k),z_k-z^*}&\leqslant O\left(\frac{\norm{z_0-z^*}^2}{k+r-1}\right).
\end{align*}
\end{theorem}

\begin{proof}
	Let $B_0=1$. By \eqref{eq:thm1assumption1} and \eqref{eq:thm1assumption2}, we have $B_k=k+r-1, A_k=\beta_k(k+r-1)(k+r)$. By Theorem \ref{thm:lyapunov}, we have\[
\E(k)\leqslant\E(0)=\frac{(r-1)r\beta_0}{2}\norm{F(z_0)}^2\leqslant\frac{(r-1)rL^2\beta_0}{2}\norm{z_0-z^*}^2.
\]
The lower bound of $\E(k)$ can be estimated as followed.\begin{align*}
	\E(k)=&\ \frac{A_k}{2}\norm{F(z_k)}^2+B_k\inner{F(z_k),z_k-z^*}+B_k\inner{F(z_k),z^*-u_k} \\
	\geqslant &\ \frac{A_k}{2}\norm{F(z_k)}^2+B_k\inner{F(z_k),z_k-z^*}\\
	&+\frac{1}{2}\left[\norm{\sqrt{\frac{A_k}{2}}F(z_k)-\sqrt{\frac{2B_k^2}{A_k}}(u_k-z^*)}^2-\frac{A_k}{2}\norm{F(z_k)}^2-\frac{2B_k^2}{A_k}\norm{u_k-z^*}^2\right]\\
	\geqslant&\ \frac{A_k}{4}\norm{F(z_k)}^2+B_k\inner{F(z_k),z_k-z^*}-\frac{B_k^2}{A_k}\norm{u_k-z^*}^2.
\end{align*}
Thus, we have\begin{align*}
	\norm{F(z_k)}^2 & \leqslant \frac{2(r-1)rL^2\beta_0}{\beta_k(k+r-1)(k+r)}\norm{z_0-z^*}^2+\frac{4}{\beta_k^2(k+r)^2}\norm{u_k-z^*}^2,\\
	\inner{F(z_k),z_k-z^*} & \leqslant \frac{(r-1)rL^2\beta_0}{2(k+r-1)}\norm{z_0-z^*}^2+\frac{1}{\beta_k(k+r)}\norm{u_0-z^*}^2.
\end{align*}
Since $\beta_k\geqslant\beta_\infty>0$, the remaining problem is showing $\norm{u_k-z}$ can be uniformly bounded by $\norm{u_0-z}=\norm{z_0-z^*}$. By triangle inequality, we have\[
\norm{u_{k+1}-z^*}\leqslant |C_k|\norm{F(z_{k+1})}+\norm{u_k-z^*}.
\]
By the definition of $C_k$ and \eqref{eq:lem1_b}, we have\[
|C_k|\leqslant\frac{r(r-1)(1-L^2\beta_0^2)}{(2r-1)L^2\beta_0^2}\cdot\frac{L^2\beta_k^3}{2(k+r-1)(1-L^2\beta_k^3)}\leqslant\frac{r(r-1)}{(4r-2)(k+r-1)}\beta_k.
\]
Also by $\sqrt{a+b}\leqslant\sqrt{a}+\sqrt{b}, \forall a>0, b>0$, we have\[
\norm{u_{k+1}-z^*}\leqslant\frac{r(r-1)}{(2r-1)[(k+r)^2-1]}\norm{u_{k+1}-z^*}+\frac{r^{\frac{3}{2}}(r-1)^{\frac{3}{2}}L\beta_0}{\sqrt{2}(2r-1)(k+r-1)(k+r)}\norm{u_0-z^*}+\norm{u_k-z^*}.
\]
Let $E_1=\dfrac{r(r-1)}{2r-1}, E_2=\dfrac{r^{\frac{3}{2}}(r-1)^{\frac{3}{2}}L\beta_0}{\sqrt{2}(2r-1)}, a_k=\norm{u_k-z^*}$. By using Lemma \ref{lem:2}, we have $\norm{u_k-z^*}\leqslant E_3\norm{z_0-z^*}$. 

In conclusion, we have\begin{align*}
	\norm{F(z_k)}^2 & \leqslant O\left(\frac{\norm{z_0-z^*}^2}{(k+r-1)(k+r)}\right), \\
	\inner{F(z_k),z_k-z^*}&\leqslant O\left(\frac{\norm{z_0-z^*}^2}{k+r-1}\right).
\end{align*}
By taking infimum respect to all $z^*\in\zero(F)$ in the first above inequality, we obtain the desired results.

\end{proof}

Since the step-size of Nesterov's accelerated gradient method can be constant, we also discuss the SEG with constant step-size. It is easy to show that if $\alpha_k=\dfrac{1}{k+r}, C_k=-\dfrac{L^2\beta^3}{(k+r-1)(1-L^2\beta_k^2)}$, $\beta_k\equiv\beta$. In conclusion, we obtain the SEG with constant step-size, described in Algorithm \ref{al:segconstant}.

\begin{algorithm}[ht]
	\label{al:segconstant}
	\caption{SEG with Constant Step-size}
	\textbf{Input: }Operator $F$, $L\in(0,+\infty)$, $ \rho\in(-\dfrac{1}{2\rho},+\infty)$\;
	\textbf{Initialization: } $z_0$, $u_0=z_0$\;
	\For{$k=0, 1, \cdots$}{
		$\tilde{z}_{k+1}=\dfrac{1}{k+r}z_k+\dfrac{k+r-1}{k+r}u_k$\;
		$z_{k+\frac{1}{2}}=\tilde{z}_{k+1}-\beta F(z_k)$\;
		$z_{k+1}=\tilde{z}_{k+1}-\beta F(z_{k+\frac{1}{2}})$\;
		$u_{k+1}=u_k+\dfrac{L^2\beta^3}{2(k+r-1)(1-L^2\beta^2)}F(z_{k+1})$\;
	}
\end{algorithm}

The convergence results of Algorithm \ref{al:segconstant} is given in Corollary \ref{coro:segconstant}.

\begin{corollary}
	\label{coro:segconstant}
	Let $\{z_k\}$, $\{z_{k+\frac{1}{2}}\}$, $\{\tilde{z}_k\}$ and $\{u_k\}$ be the sequences generated by Algorithm \ref{al:segconstant} with $\beta\in(0,\dfrac{1}{L})$. Then we have\begin{align*}
		\norm{F(z_k)}^2 & \leqslant O\left(\frac{\dist(z_0,\zero(F))^2}{(k+r-1)(k+r)}\right), \\
		\inner{F(z_k),z_k-z^*}&\leqslant O\left(\frac{\norm{z_0-z^*}^2}{k+r-1}\right),\quad\forall z^*\in\zero(F).
	\end{align*}
\end{corollary}

\section{Discussion of Stochastic SEG+}
\subsection{Symplectic Stochastic Extragradient+ Method}
In this section, we consider the following stochastic inclusion problem:\[
0=F(z)=\mathop{\mathbb{E}}_{\zeta}[F(z;\zeta)].
\]
Such stochastic inclusion problem has been wildly used to train GANs and train Adversarial training deep neural network classifiers. To solve the above problem, we proposed the following stochastic SEG+ framework.
\begin{subequations}
	\begin{align}
		\tilde{z}_{k+1}&=(1-\alpha_k) z_k+\alpha_ku_k,\label{eq:ssega}\\
		z_{k+\frac{1}{2}}&=\tilde{z}_{k+1}-(1-\alpha_k)\beta_k[F(z_k)+\zeta_k],\label{eq:ssegb}\\
		z_{k+1}&=\tilde{z}_{k+1}-\beta_k[F(z_{k+\frac{1}{2}})+\zeta_{k+\frac{1}{2}}],\label{eq:ssegc}\\
		u_{k+1}&=u_k-C_k[F(z_{k+1})+\zeta_{k+1}].\label{eq:ssegd}
	\end{align}
\end{subequations}
\begin{theorem}
	\label{thm:ssegrate}
	Let $\{z_{\frac{k}{2}}\}$, $\{\tilde{z}_{k+1}\}$ and $\{u_k\}$ be the sequences generated by \eqref{eq:ssega}-\eqref{eq:ssegd}. If $F$ is $L-$Lipschitz continuous and monotone, the parameters satisfy the following requirements:
	\begin{align*}
		\beta_k&\in (0,\frac{1}{L}),\\
		\frac{B_k\beta_k}{\alpha_k} & = \left(\frac{B_k\beta_k}{\alpha_k}+\frac{B_k\beta_k}{1-\alpha_k}\right)L^2\beta_k^2(1-\alpha_k), \\
		\frac{B_k\beta_k}{2\alpha_k}+\frac{B_k\beta_k}{2-2\alpha_k}&\geqslant B_{k+1}C_k+\frac{C_k(1-\alpha_k)B_{k+1}-C_kB_k}{2-2\alpha_k}+A_{k+1},\\
		\frac{B_k\beta_k}{2\alpha_k}&\geqslant A_k,\quad			\frac{(1-\alpha_k)B_{k+1}-B_k}{2C_k(1-\alpha_k)}\text{ is non-increasing},
	\end{align*}
	then we have
	\begin{align*}
		\expect{\E(k)}\leqslant&\  \expect{\E(0)}+\sum_{i=0}^{k-1}\left(\frac{B_i\beta_i}{1-\alpha_i}+\frac{B_i\beta_i}{\alpha_i}\right)L\beta_i\sigma_{i+\frac{1}{2}}^2+\frac{C_iB_i-C_k(1-\alpha_i)B_{i+1}}{2-2\alpha_i}\sigma_{i+1}^2\\
		&\ +\sum_{i=0}^{k-1}\left(\frac{B_i\beta_i}{2\alpha_i}+\frac{B_i\beta_i}{2-2\alpha_i}\right)[L^2\beta_i^2\sigma_{i+\frac{1}{2}}^2+ L^2\beta_i^2(1-\alpha_i)^2\sigma_i^2+2L^3\beta_i^3(1-\alpha_i)^2\sigma_i^2]. 
	\end{align*}
\end{theorem}

\begin{algorithm}
	\label{al:sseg}
	\caption{Symplectic Stochastic Extragradient+ Method, SSEG+}
	\textbf{Input: }Operator $F$,  $L\in(0,+\infty), \rho\in(-\dfrac{1}{2\rho},+\infty)$\;
	\textbf{Initialization: }$z_0, u_0=z_0$\;
	\For{$k=0, 1, \cdots$}{
		$\tilde{z}_{k+1}=\dfrac{k}{k+r}z_k+\dfrac{r}{k+r}u_k$\;
		$z_{k+\frac{1}{2}}=\tilde{z}_{k+1}-\dfrac{k}{k+r}\dfrac{1}{L}[F(z_k)+\zeta_k]$\;
		$z_{k+1}=\tilde{z}_{k+1}-\dfrac{1}{L}[F(z_{k+\frac{1}{2}})+\zeta_{k+\frac{1}{2}}]$\;
		$u_{k+1}=u_k-\dfrac{D}{r}[F(z_{k+1})+\zeta_{k+1}]$.
	}	
\end{algorithm}
\begin{corollary}
	\label{coro:sseg+}
	Let $\{z_{\frac{k}{2}}\}$, $\{\tilde{z}_{k+1}\} $ and $\{u_k\}$ be the sequences generated by Algorithm \ref{al:sseg}. If $F$ is $L-$Lipschitz continuous and monotone, $r\geqslant 2$, $0<D\leqslant 1$, and there exists two positive constant $E_1$ and $E_2$ such that $\sigma_k^2\leqslant E_1\dfrac{\varepsilon}{rk}$, $\sigma_{k+\frac{1}{2}}^2\leqslant E_2\dfrac{\varepsilon}{r(k+r)}$, there exists a positive constant $E_3$ such that
	\begin{align*}
		\expect{\norm{F(z_k)}^2}\leqslant\frac{L(r^3-r^2)}{Dk^2}\dist(z_0,\zero(F))^2+E_3\varepsilon.
	\end{align*}
\end{corollary}

\subsection{Proof of Theorem \ref{thm:ssegrate}}
Before proving Theorem \ref{thm:ssegrate}, we need the following lemma.
\begin{lemma}
	\label{lem:sseglemma}
	Let $\{z_k\}, \{u_k\}$ be the sequences generated by \eqref{eq:ssega}-\eqref{eq:ssegd}. If the random sequence $\{\zeta_{\frac{k}{2}}\}$ satisfies $\expect{\zeta_{\frac{k}{2}}}=0$, $ \expect{\zeta_{\frac{k}{2}}^2}=\sigma_{\frac{k}{2}}^2$ for all $k\geqslant 0$, then we have\begin{align*}
		\left|\expect{F(z_{k+1}),\zeta_{k+\frac{1}{2}}}\right| &\leqslant L\beta_k\sigma_{k+\frac{1}{2}}^2,  \\
		\left|\expect{F(z_{k+\frac{1}{2}}),\zeta_k}\right| & \leqslant L(1-\alpha_k)\beta_k\sigma_k^2.
	\end{align*}
\end{lemma}

The proof for Lemma \ref{lem:sseglemma} is analogous to the proof for Lemma D.2 in \cite{lee21}. We refer the proof to \cite{lee21}.

\noindent\textbf{The proof of Theorem \ref{thm:ssegrate}}\quad Consider the following Lyapunov function\begin{equation}
	\E(k)=A_k\norm{F(z_k)}^2+B_k\inner{F(z_k),z_k-u_k}+\frac{B_k-(1-\alpha_k)B_{k+1}}{2(1-\alpha_k)C_k}\norm{u_k-z^*}^2,
\end{equation}
where $z^*$ is the zero-point of $F$, i. e. $F(z^*)=0$. We show that $\{\expect{\E(k)}\}$ is non-increasing.

Step 1: Divide the difference $\expect{\E(k+1)}-\expect{\E(k)}$ into three parts. 	
\begin{align*}
	&\expect{\E(k+1)}-\expect{\E(k)}\\
	=&\underbrace{A_{k+1}\expect{\norm{F(z_{k+1})}^2}-A_k\expect{\norm{F(z_k)}^2}}_{\text{I}}\\
	&+\underbrace{B_{k+1}\expect{\inner{F(z_{k+1}),z_{k+1}-u_{k+1}}}-B_k\expect{\inner{F(z_k),z_k-u_k}}}_{\text{II}} \\
	&+\underbrace{\frac{B_{k+1}-(1-\alpha_{k+1})B_{k+2}}{2(1-\alpha_{k+1})C_{k+1}}\expect{\norm{u_{k+1}-z^*}^2}-\frac{B_k-(1-\alpha_k)B_{k+1}}{2(1-\alpha_k)C_k}\expect{\norm{u_k-z^*}^2}}_{\text{III}}.
\end{align*}

Step 2: Reckon the upper bound of all three parts. First we consider II.
\begin{align*}
	\text{II}=&(B_{k+1}-B_k)\expect{\inner{F(z_{k+1}),z_{k+1}-u_{k+1}}}+B_k\expect{\inner{F(z_{k+1}),z_{k+1}-z_k-u_{k+1}+u_k}} \\
	&+B_k\expect{\inner{F(z_{k+1})-F(z_k),z_k-u_k}}.
\end{align*}
Since $\expect{\zeta_{k+1}|z_{k+1}}=0$, we have\[
\expect{F(z_{k+1}),\zeta_{k+1}}=0.
\]
Combining the above equation, \eqref{eq:ssegc} and  \eqref{eq:ssegd}, we have
\begin{equation}
	\begin{split}
		&\ B_k\expect{\inner{F(z_{k+1}),z_{k+1}-z_k-u_{k+1}+u_k}}\\
		=&\ B_k\expect{\inner{F(z_{k+1}),\frac{\alpha_k}{1-\alpha_k}(u_k-z_{k+1})-\frac{\beta_k}{1-\alpha_k}[F(z_{k+\frac{1}{2}})+\zeta_{k+\frac{1}{2}}]}}+B_kC_k\expect{\norm{F(z_{k+1})}^2}\\
		\leqslant&B_k\expect{\inner{F(z_{k+1}),\frac{\alpha_k}{1-\alpha_k}(u_k-z_{k+1})-\frac{\beta_k}{1-\alpha_k}F(z_{k+\frac{1}{2}})}}+\frac{B_kL\beta_k^2}{1-\alpha_k}\sigma^2_{k+\frac{1}{2}}+B_kC_k\expect{\norm{F(z_{k+1})}^2}.\\
	\end{split}
	\label{eq:thm3_a}
\end{equation}
In addition, because of \eqref{eq:ssegb}, we have\begin{equation}
	\begin{split}
		&\ B_k\expect{\inner{F(z_{k+1})-F(z_k),z_k-u_k}}\\
		=&\ B_k\expect{\inner{F(z_{k+1})-F(z_k),-\frac{1}{\alpha_k}(z_{k+1}-z_k)-\frac{\beta_k}{\alpha_k}[F(z_{k+\frac{1}{2}})+\zeta_{k+\frac{1}{2}}]}}\\
		\leqslant&\  B_k\expect{\inner{F(z_{k+1})-F(z_k),-\frac{1}{\alpha_k}(z_{k+1}-z_k)-\frac{\beta_k}{\alpha_k}F(z_{k+\frac{1}{2}})}}+\frac{B_kL\beta_k^2\sigma_{k+\frac{1}{2}}^2}{\alpha_k}.
	\end{split}
	\label{eq:thm3_b}
\end{equation}
Also,
\begin{equation}
	\label{eq:thm3_c}
	\begin{split}
		&\ (B_{k+1}-B_k)\expect{\inner{F(z_{k+1}),z_{k+1}-u_{k+1}}}\\
		=&\ (B_{k+1}-B_k)C_k\expect{\norm{F(z_{k+1})}^2}+(B_{k+1}-B_k)\expect{\inner{F(z_{k+1}),z_{k+1}-u_k+C_k\zeta_{k+1}}}\\
		=&\ (B_{k+1}-B_k)C_k\expect{\norm{F(z_{k+1})}^2}+(B_{k+1}-B_k)\expect{\inner{F(z_{k+1}),z_{k+1}-u_k}}.
	\end{split}
\end{equation}
By summing \eqref{eq:thm3_a}, \eqref{eq:thm3_b} and \eqref{eq:thm3_c}, we have
\begin{align*}
	\text{II}\leqslant&\ B_k\expect{\inner{F(z_{k+1}),\frac{\alpha_k}{1-\alpha_k}(u_k-z_{k+1})-\frac{\beta_k}{1-\alpha_k}F(z_{k+\frac{1}{2}})}}\\
	&\ +B_k\expect{\inner{F(z_{k+1})-F(z_k),-\frac{1}{\alpha_k}(z_{k+1}-z_k)-\frac{\beta_k}{\alpha_k}F(z_{k+\frac{1}{2}})}}+\left(\frac{B_k\beta_k}{1-\alpha_k}+\frac{B_k\beta_k}{\alpha_k}\right)L\beta_k\sigma_{k+\frac{1}{2}}^2\\
	&\ +B_{k+1}C_k\expect{\norm{F(z_{k+1})}^2}+(B_{k+1}-B_k)\expect{\inner{F(z_{k+1}),z_{k+1}-u_k}}.
\end{align*}
Since $F$ is $L-$Lipschitz continuous, we have\[
\norm{F(z_{k+1})-F(z_{k+\frac{1}{2}})}^2\leqslant L^2\norm{z_{k+1}-z_{k+\frac{1}{2}}}^2.
\]
Due to \eqref{eq:ssegb} and \eqref{eq:ssegc}, we have\[
\norm{z_{k+1}-z_{k+\frac{1}{2}}}^2=\beta_k^2\norm{F(z_{k+\frac{1}{2}})+\zeta_{k+\frac{1}{2}}-(1-\alpha_k)[F(z_k)+\zeta_k]}^2.
\]
By using Lemma \ref{lem:sseglemma}, we have\begin{align*}
	\expect{\norm{F(z_{k+1})-F(z_{k+\frac{1}{2}})}^2}  \leqslant&\  L^2\beta_k^2\expect{\norm{F(z_{k+\frac{1}{2}})}^2}+2L^2\beta_k^2(1-\alpha_k)\expect{\inner{F(z_{k+\frac{1}{2}}),F(z_k)}}+L^2\beta_k^2\sigma_{k+\frac{1}{2}}^2\\
	&\ + L^2\beta_k^2(1-\alpha_k)^2\sigma_k^2+L^2\beta_k^2(1-\alpha_k)^2\expect{\norm{F(z_k)}^2}+2L^3\beta_k^3(1-\alpha_k)^2\sigma_k^2.
\end{align*}
Because of the following equality\[
-\inner{F(z_{k+1}),F(z_{k+\frac{1}{2}})}=\frac{1}{2}\left[\norm{F(z_{k+1})-F(z_{k+\frac{1}{2}})}^2-\norm{F(z_{k+1})}^2-\norm{F(z_{k+\frac{1}{2}})}^2\right]
\]
and the assumption\[
\frac{B_k\beta_k}{\alpha_k}  = \left(\frac{B_k\beta_k}{\alpha_k}+\frac{B_k\beta_k}{1-\alpha_k}\right)L^2\beta_k^2(1-\alpha_k),
\]
we obtain the following estimation of II:
\begin{align*}
	\text{II}\leqslant&\  -\left(\frac{B_k\beta_k}{2\alpha_k}+\frac{B_k\beta_k}{2-2\alpha_k}-B_{k+1}C_k\right)\expect{\norm{F(z_{k+1})}^2}-\left(\frac{B_k\beta_k}{2\alpha_k}+\frac{B_k\beta_k}{2-2\alpha_k}\right)(1-L^2\beta_k^2)\expect{\norm{F(z_{k+\frac{1}{2}})}^2}\\
	&\ -\frac{B_k\beta_k}{2\alpha_k}\expect{\norm{F(z_k)}^2}-\frac{B_k}{1-\alpha_k}\expect{\inner{F(z_{k+1})-F(z_k),z_{k+1}-z_k}}\\
	&\ +\underbrace{\left(B_{k+1}-\frac{B_k}{1-\alpha_k}\right)\expect{\inner{F(z_{k+1}),z_{k+1}-u_k}}}_{\text{II}_1}+\left(\frac{B_k\beta_k}{1-\alpha_k}+\frac{B_k\beta_k}{\alpha_k}\right)L\beta_k\sigma_{k+\frac{1}{2}}^2\\
	&\ +\left(\frac{B_k\beta_k}{2\alpha_k}+\frac{B_k\beta_k}{2-2\alpha_k}\right)[L^2\beta_k^2\sigma_{k+\frac{1}{2}}^2+ L^2\beta_k^2(1-\alpha_k)^2\sigma_k^2+2L^3\beta_k^3(1-\alpha_k)^2\sigma_k^2].
\end{align*}

Now estimate III. Utilizing \eqref{eq:normdifference}, we have\begin{align*}
	\text{III}\leqslant&\ \frac{B_k-(1-\alpha_k)B_{k+1}}{2(1-\alpha_k)C_k}\expect{\norm{u_{k+1}-z^*}^2-\norm{u_k-z^*}^2}\\
	=&\ \underbrace{\frac{B_k-(1-\alpha_k)B_{k+1}}{1-\alpha_k}\expect{\inner{F(z_{k+1}),u_k-z^*}}}_{\text{III}_1}+\frac{C_kB_k-C_k(1-\alpha_k)B_{k+1}}{2-2\alpha_k}\left\{\expect{\norm{F(z_{k+1})}^2}+\sigma_{k+1}^2\right\}.
\end{align*}

Step 3: Deduce the upper bound of $\E(k+1)-\E(k)$. Noticing that $\text{II}_1+\text{III}_1=\dfrac{(1-\alpha_{k})B_{k+1}-B_k}{1-\alpha_k}\inner{F(z_{k+1}),z_{k+1}-z^*}$, we can simplify the upper bound of $\expect{\E(k+1)}-\expect{\E(k)}$ as follows:\begin{align*}
	&\ \expect{\E(k+1)}-\expect{\E(k)}\\
	\leqslant&\ -\left(\frac{B_k\beta_k}{2\alpha_k}+\frac{B_k\beta_k}{2-2\alpha_k}-B_{k+1}C_k-\frac{C_kB_k-C_k(1-\alpha_k)B_{k+1}}{2-2\alpha_k}-A_{k+1}\right)\expect{\norm{F(z_{k+1})}^2}\\
	&\ -\left(\frac{B_k\beta_k}{2\alpha_k}+\frac{B_k\beta_k}{2-2\alpha_k}\right)(1-L^2\beta_k^2)\expect{\norm{F(z_{k+\frac{1}{2}})}^2}-\left(\frac{B_k\beta_k}{2\alpha_k}-A_k\right)\expect{\norm{F(z_k)}^2}\\
	&\ -\frac{B_k}{1-\alpha_k}\expect{\inner{F(z_{k+1})-F(z_k),z_{k+1}-z_k}}+\left(B_{k+1}-\frac{B_k}{1-\alpha_k}\right)\expect{\inner{F(z_{k+1}),z_{k+1}-z^*}}\\
	&\ +\left(\frac{B_k\beta_k}{1-\alpha_k}+\frac{B_k\beta_k}{\alpha_k}\right)L\beta_k\sigma_{k+\frac{1}{2}}^2+\frac{C_kB_k-C_k(1-\alpha_k)B_{k+1}}{2-2\alpha_k}\sigma_{k+1}^2\\
	&\ +\left(\frac{B_k\beta_k}{2\alpha_k}+\frac{B_k\beta_k}{2-2\alpha_k}\right)[L^2\beta_k^2\sigma_{k+\frac{1}{2}}^2+ L^2\beta_k^2(1-\alpha_k)^2\sigma_k^2+2L^3\beta_k^3(1-\alpha_k)^2\sigma_k^2].
\end{align*}
\hfill$\Box$

\bibliographystyle{amsplain}
\bibliography{reference}

\end{document}